\documentclass[reqno]{amsart}

\usepackage{amsmath,amssymb,amsthm}

\newtheorem{thm}{Theorem}[section]

\newtheorem{lem}{Lemma}[section]
\newtheorem{pro}{Proposition}[section]
\newtheorem{cor}{Corollary}[section]
\theoremstyle{definition}
\newtheorem{rem}{Remark}[section]

  \makeatletter
  \@addtoreset{equation}{section}
  \makeatother

\title[On the generalized Fornberg-Whitham equation]
{Wave-breaking phenomena and global existence for the generalized Fornberg-Whitham equation}

\author{Kenta Itasaka}

\address{Department of Mathematics, Hokkaido University, Sapporo, 060-0810, Japan}

\email{s163003@math.sci.hokudai.ac.jp}

\begin{document}

\begin{abstract}

We consider the generalized Fornberg-Whitham equation.
We study sufficient conditions for blow-up of solutions and
show the global existence with small initial data.
Also we give some relations to the Burgers equation.

\end{abstract}

\maketitle

\section{Introduction}
We consider the initial value problem for the generalized Fornberg-Whitham equation 
\begin{align}
\label{eq1}
\left \{
\begin{aligned}
& \partial_t u + \frac{1}{p}\partial_x [u^p] + \int _{\mathbf{R}} B \mathrm{e}^{-b|x-\xi|}u_{\xi}(t,\xi) \, d \xi =0,
\quad (t,x) \in \mathbf{R}_+ \times \mathbf{R}, \\
& u(0,x)=u_0(x), 
\end{aligned}
\right.
\end{align}
where $ p \ge 2$, $p\in \mathbf{N}$, $ B >0$ and $ b >0$.
If $p=2$, $B=\frac{1}{2}$ and $b=\frac{3}{2}$, the equation in (\ref{eq1}) becomes the Fornberg-Whitham equation.
It was derived by Whitham \cite{whitham1967variational} and by Whitham and Fornberg \cite{fornberg1978numerical} 
as a model for so-called "breaking waves".



The aims of this paper are to study sufficient conditions for blow-up of solutions in the case $p=2$ and 
to show the global existence with small initial data for $p \ge 5$.
We begin by discussing each case more precisely.


First, we treat the case of $ p=2 $, which is  
\begin{align}
\label{eq2}
\left \{
\begin{aligned}
& \partial_t u + u \partial _x u + \int _{\mathbf{R}} B \mathrm{e}^{-b|x-\xi|}u_{\xi}(t,\xi) \, d \xi =0, 
\quad (t,x) \in \mathbf{R}_+ \times \mathbf{R}, \\
& u(0,x)=u_0(x).
\end{aligned}
\right.
\end{align}
This equation can be regarded as the Whitham-type equation 
\begin{align}
\label{eq10}
\left \{
\begin{aligned}
& \partial_t u + u \partial_xu + \int _{\mathbf{R}} K(x-\xi)u_{\xi}(t,\xi) \, d \xi =0, 
\quad (t,x) \in \mathbf{R}_+ \times \mathbf{R}, \\
& u(0,x)=u_0(x).
\end{aligned}
\right.
\end{align}
Indeed, taking $K(x)=B \mathrm{e}^{-b|x|}$ in \eqref{eq10}, we obtain the equation \eqref{eq2}.
The equation \eqref{eq10} is known as a model for describing breaking waves.

Wave-breaking phenomena for \eqref{eq10} was first studied by Constantin and Escher \cite{constantin1998wave}.
In \cite{ma2016wave}, their result was improved by Ma, Liu and Qu.
Recently, Haziot \cite{haziot2017wave} obtained a blow-up condition only for the Fornberg-Whitham equation.
We summarize these blow-up conditions for \eqref{eq2} as follows:
\begin{itemize}
\item Constantin and Eshcer \cite{constantin1998wave}
\begin{align}
\label{ceq}
\inf_{x\in \mathbf{R}} u_0'(x) + \sup_{x\in \mathbf{R}} u_0'(x) \le -2B.
\end{align}
\item Ma, Liu and Qu \cite{ma2016wave}
\begin{align}
\label{meq}
\inf_{x\in \mathbf{R}} u_0'(x) < \min \left \{ {-2B, \, \frac{-B-\sqrt{B^2+4B \sup_{x\in \mathbf{R}} u_0'(x)}}{2}}\right \}.
\end{align}
\item Haziot \cite{haziot2017wave}
\begin{align}
\label{heq}
5\inf_{x\in \mathbf{R}} u_0'(x) + \sup_{x\in \mathbf{R}} u_0'(x) \le -6B.
\end{align}
\end{itemize}
If initial data $ u_0$ satisfies \eqref{ceq}, \eqref{meq} or \eqref{heq}, the corresponding solution to \eqref{eq2} blows up in finite time.

Roughly speaking, letting $ B \rightarrow 0$ or $b \rightarrow 0$ in \eqref{eq2}, 
we obtain the Burgers equation
\begin{align}
\label{eq7}
\left \{
\begin{aligned}
& \partial_t u + u \partial _x u =0, \quad (t,x) \in \mathbf{R}_+ \times \mathbf{R},  \\
& u(0,x)=u_0(x).
\end{aligned}
\right.
\end{align}
As is well known, 
if $u_0'(x_0)<0$ for some $ x_0 \in \mathbf{R}$, the corresponding solution of \eqref{eq7} blows up in finite time.
Therefore, it is to be expected that 
when $ B $ or $ b $ is sufficiently small,
the corresponding solution of \eqref{eq2} blows up in finite time.
The expected result for $B$ follows from the above blow-up conditions.
However, $ b $ does not appear there,
and hence the expected result for $b$ does not follow from those.

To solve this problem, we show the following new blow-up condition, which is the heart of the matter.
\begin{thm}
\label{thm1}
Assume that $s\ge3$ and $u_0 \in H^s$.
Also, let $B>0$ and $b>0$, and set
\begin{align*}
F(t,x_0)= 2 B  b  u_0(x_0)  + B b^{\frac{3}{2}} \| u_0 \|_{L^2} +2B^2b^{\frac{3}{2}}\| u_0 \|_{L^2}t.
\end{align*}
If there exist $ x_0 \in \mathbf{R} $ and $  T>0 $ such that $  F(T,x_0) \ge 0$ and 
\begin{align}
\label{eqthm1.1}
u_0'(x_0) \le - \alpha \left(\frac{F(T,x_0)^{\frac{1}{4}}+ \sqrt{F(T,x_0)^{\frac{ 1}{2}}
+\frac{1+\frac{1}{\alpha}}{1-\frac{1}{2\alpha}} \, \frac{4}{T}  }}{2} \right)^2
\end{align}
for some $\alpha \ge 1$, then the solution of \eqref{eq2} with initial data $ u_0 $ blows up in finite time.
Moreover, for the blow-up time $T_0$, we obtain the upper bound by
\begin{align*}
T_0 \le - \frac{1+\frac{1}{\alpha}}{1-\frac{1}{2\alpha}} \, \frac{1}{u_0'(x_0) + \sqrt{-u_0'(x_0)} F(T,x_0)^{\frac{1}{4}} } \le T.
\end{align*}
\end{thm}
The proof is adapted from the method in \cite[Theorem 3.2]{ma2016wave}. 
However, since $ \partial_{xx} \left[B \mathrm{e}^{-b|\cdot|} \right]\notin L^2$, more complicated assumptions rather than \cite{ma2016wave}
are made.

\begin{rem}
This theorem implies the expected result for $b$ as stated before.
Indeed, let $u_0'(x_0)<0$ and $T> - \frac{4}{u_0'(x_0)}$. 
Then for $b$ sufficiently small, the requirment \eqref{eqthm1.1} is saitsfied, in particular when $\alpha=1$.
\end{rem}

\begin{rem}
Let $\lambda\ge1$ and $u_0(x)=\mathrm{e}^{-\lambda x^2}$.  
Then we have 
\begin{align*}
\|u_0\|_{L^2}=2^{-\frac{1}{4}}\pi^{\frac{1}{4}}\lambda^{-\frac{1}{4}}, \quad 
u_0 \left(\frac{1}{\sqrt{\lambda}} \right)=\mathrm{e}^{-1}  \quad \mathrm{and} \quad u_0'\left(\frac{1}{\sqrt{\lambda}}\right) 
= -2 \mathrm{e}^{-1} \sqrt{\lambda}.
\end{align*}
Setting $ x_0= \frac{1}{\sqrt{\lambda}} $ and
$T= \frac{\lambda^{\frac{1}{4}}}{B^{\frac{2}{3}}b^{\frac{1}{2}}} $, we obtain
\begin{align*}
\left(\frac{F(T,x_0)^{\frac{1}{4}}+ \sqrt{F(T,x_0)^{\frac{ 1}{2}}
+  \, \frac{16}{T}  }}{2} \right)^2 \le C (B^{\frac{1}{2}} b^{\frac{1}{2}} + B^{\frac{1}{2}} b^{\frac{3}{4}} + B^{\frac{2}{3}} b^{\frac{1}{2}}),
\end{align*}
where $C$ is independent of $ \lambda$, $B$ and $b$.
Therefore if $ \lambda \ge C (B b + B b^{\frac{3}{2}} + B^{\frac{4}{3}} b +1) $,
the requirment \eqref{eqthm1.1} is saitsfied when $\alpha=1$, and hence the corresponding solution of \eqref{eq2}
blows up in finite time.
\end{rem}

From Theorem \ref{thm1}, we can see a more detailed relation 
between 
\eqref{eq2} and 
\eqref{eq7}.
\begin{thm}
\label{thm10}
Assume that $s\ge3$ and $u_0 \in H^s$.
Also, let $ T_{max}^{B,b} $ be the maximal existence time of the corresponding solution $u^{B,b}$ to (\ref{eq2}). Then
\begin{align}
\lim_{B \rightarrow 0 \, \mathrm{or} \, b \rightarrow 0} T_{max}^{B,b} = -\frac{1}{\inf_{x \in \mathbf{R}}u_0'(x)}. \notag
\end{align}
\end{thm}
It is known that the maximal existence time of the solution to \eqref{eq7} 
is 
\begin{align}
-\frac{1}{\inf_{x \in \mathbf{R}}u_0'(x)}.
\end{align} 
For this reason, Theorem \ref{thm10} is what is naturally expected.

To describe our next result, we define $G(u_0), \, B(u_0) \subset \mathbf{R}_+ \times \mathbf{R}_+$ as
\begin{center}
$G(u_0)$=\{$B>0$, $b>0$  :  the solution $u^{B,b}$ exists globally\},  

$ B(u_0)$=\{$B>0$, $b>0$  : the solution $u^{B,b}$ blows up in finite time\}.
\end{center}
In \cite{tanaka2013}, Tanaka showed numerically that \eqref{eq2} has global solution and blow-up solution
depending on the parameter $(B,b) $.
Moreover, he attempted to dive the boundary layer in a diagram $(B,b) \in \mathbf{R}_+ \times \mathbf{R}_+ $
into $G(u_0)$ and $B(u_0)$.
However, this question is still unsolved.


On this question, we obtain the following result.
\begin{thm}
\label{thm11}
Let $s \ge3$ and $u_0 \in H^s$.
%
Then $G(u_0)$ is a closed set of $\mathbf{R}^2$, and $B(u_0)$ an open set of $\mathbf{R}^2$.  
\end{thm}
 The proof is based on Theorem \ref{thm1} and the continuous dependence on $(B,b)$, which will be proved later.
Theorem \ref{thm11} shows that the boundary is contained in $G(u_0)$.



Next, we deal with the case of $ p\ge5 $.
Let us consider the more general equation
\begin{align}
\left \{
\begin{aligned}
\label{eq6.1}
& \partial_t u + \frac{1}{p} \partial_x [u^p] + \mathcal{F}^{-1}[i \, m( \xi ) \widehat{u}]=0, 
\quad (t,x) \in \mathbf{R}_+ \times \mathbf{R}, \\
& u(0,x)=u_0(x), 
\end{aligned}
\right.
\end{align}
where $ p \in \mathbf{N}$, $ p \ge 2$ and $m(\xi) \in \mathbf{R}$.
If $ m(\xi)=\frac{2Bb \, \xi}{b^2+\xi^2}$, the equation in \eqref{eq6.1} becomes the one in \eqref{eq1}.
Here we rewtite \eqref{eq6.1} into the integral equation in order to state our assumptions:
\begin{align}
& u= T(t)u_0 + \frac{1}{p} \int^t_0 T(t-s)\partial_x[u^p] ds = 0, \notag 
\end{align}
where 
\begin{align*}
T(t)u_0=\mathcal{F}^{-1}[e^{-i \, t \, m(\xi)}\widehat{u_0}].
\end{align*}
In the following theorem, we assume that 
\begin{align}
\mathrm{(H)} \, \left \{
\begin{aligned}
& \, \|T(t)u_0\|_{L^2} = \|u_0\|_{L^2}, \notag \\
& \, \|T(t)u_0\|_{L^r} \le C \, t^{-\frac{1}{n}(1-\frac{2}{r})}\|u_0\|_{W^{\gamma(1-\frac{2}{r}),r'}}, \quad (2<r<\infty), \notag
\end{aligned}
\right.
\end{align}
for some $ n \in \mathbf{N}$ and $\gamma>0$.
We also suppose that \eqref{eq6.1} is local well-posed in $H^{\frac{5}{2}+\mathrm{max}\{\frac{1}{2},\gamma\}}$,
and the existence time $T=T(u_0)$ of local solution 
depends only on $\|u_0\|_{H^{\frac{5}{2}+\mathrm{max}\{\frac{1}{2},\gamma\}}}$
such that 
\begin{align*}
\liminf_{\|u_0\|_{H^{\frac{5}{2}+\mathrm{max}\{\frac{1}{2},\gamma\}}} \rightarrow 0} T(u_0)>0.
\end{align*}
\begin{thm}
\label{thm7}
%
%
%
%
Assume that the condition $\mathrm{(H)}$ holds.
Also, let $ p^* = \left[ \frac{n+\sqrt{n^2+4n}}{2}+2 \right], $ where $[ \, \cdot \, ]$ is the Gauss symbol.
Then, for integer $p \ge p^*$, there exists $ \epsilon = \epsilon(p)>0$ such that 
if $ \|u_0\|_{H^{\frac{5}{2}+\mathrm{max}\{\frac{1}{2},\gamma\}}}+\|u_0\|_{W^{\frac{3}{2}+ \gamma,1 }}< \epsilon $,
\eqref{eq6.1} has a global solution. In addtion, 
\begin{align}
\sup_{0<t<\infty}\|u(t)\|_{H^{\frac{5}{2}+\mathrm{max}\{\frac{1}{2},\gamma\}}} \le 4 \epsilon. \notag
\end{align}
\end{thm}
 The proof is adapted from that of \cite[Theorem 2]{stefanov2010well}. 
The cace of $n=2$ and $\gamma=\frac{3}{2}$ is proved there.

As we will see later, the solutions to the linearized equation of \eqref{eq1} decays at the rate of $t^{-\frac{1}{3}}$.
Hence, we obtain the following result.
\begin{thm}
\label{thm8}
For integer $p \ge 5$, there exists $ \epsilon = \epsilon(p)>0$ such that 
if $ \|u_0\|_{H^{4}}+\|u_0\|_{W^{3,1}}< \epsilon $,
\eqref{eq1} has a global solution. In addtion, $ \sup_{0<t<\infty}\|u(t)\|_{H^4} \le 4 \epsilon.$
\end{thm}

This paper is organized as follows. Section \ref{sect} presents some preliminaries.
In Section \ref{sec3}, we prove Theorem \ref{thm1} and establish some results on the blow-up.
In Section \ref{sec4}, we first give a lower bound of the maximal existence time.
Then combining this result with Theorem \ref{thm1}, we verify Theorem \ref{thm10}.
In Section \ref{sec5}, we study the continuous dependence on $(B,b)$ 
and give the proof of Theorem \ref{thm11}.
Section \ref{sec6} establishes Theorem \ref{thm7}.
In Section \ref{sec7}, we show that
the solutions to the linearized equation of \eqref{eq1} decays at the rate of $t^{-\frac{1}{3}}$.
%
\section{preliminaries}
\label{sect}

Fourier transform and its inverse are denoted by $ \, \widehat{} \, $ and $\mathcal{F}^{-1}$, respectively.
We define the Sobolev norms $\| \, \cdot \, \|_{W^{s,p}}$ as
\begin{align}
\| f \|_{W^{s,p}} := \| \mathcal{F}^{-1} [(1+|\xi|^2)^{\frac{s}{2}} \widehat{f} \,] \|_{L^p}. \notag
\end{align}
When $p=2$, we write $ H^s$ instead of $W^{s,2}$.

Recently, Holmes and Thompson \cite{holmes2017well} obtaind the local well-posedness and a blow-up criterion for the Fornberg-Whitham equation.
By using the same argument, we can verify the following two propositions.
\begin{pro}
\label{pro1}
Assume that $ s> \frac{3}{2}$ and $ u_0 \in H^s$. 
Then there exists a time $ T>0 $ and a unique solution $ u \in C([0,T]; H^s) \cap C^1([0,T]; H^{s-1})$.
Also, we can take $T=\min \left \{\frac{C}{(1+B)\| u_0 \|_{H^s}}, \, \frac{C}{1+B}\right \}$ at least 
and then 
\begin{align}
\| u(t) \|_{H^s} \le 2 \| u_0 \|_{H^s}, \quad t \in [0,T]. \notag
\end{align}
\end{pro}
\begin{proof}
%
We consider the following linear transport equation:
\begin{equation*}
\left \{
\begin{aligned}
& \partial_t u^{n+1} + u^n \partial _x u^{n+1} + \int _{\mathbf{R}} B \mathrm{e}^{-b|x-\xi|}u^n_{\xi}(t,\xi) \, d \xi =0, 
\quad (t,x) \in \mathbf{R}_+ \times \mathbf{R}, \\
& u^{n+1}(0,x)=J_{n+1}u_0(x),
\end{aligned}
\right.
\end{equation*}
where $n\in \mathbf{Z}_+$, $u^0(t,x)=0$, and $J_{n+1}$ is the mollifier.
By \cite[Proposition A.1]{danchin2001few}, we have for the solution of the above equation
\begin{align*}
&\|u^{n+1}(t)\|_{H^s} \\
& \le \mathrm{e}^{CU^n(t)}\left(\|J_{n+1}u_0\|_{H^s}+C \int_0^t \mathrm{e}^{-CU^n(\tau)} 
\left \|\int _{\mathbf{R}} B \mathrm{e}^{-b| \cdot -\xi|}u^n_{\xi} d\xi \right \|_{H^s} d \tau \right),
\end{align*}
where $U^n(t)= \int_0^t \| \partial _x u^n(\tau) \|_{H^{s-1}}$.
It is well known that $\|J_{n+1}u_0\|_{H^s} \le \|u_0\|_{H^s}$.
Since
\begin{align*}
\int _{\mathbf{R}} B \mathrm{e}^{-b| \cdot -\xi|}u^n_{\xi}(t,\xi) = \mathcal{F}^{-1} \left[ \frac{i \, 2Bb \xi}{b^2+\xi^2} \widehat{u^n} \,\right],
\end{align*}
we have
\begin{align*}
\left \|\int _{\mathbf{R}} B \mathrm{e}^{-b| \cdot -\xi|}\partial_{\xi}u^n(t,\xi) d\xi \right \|_{H^s}
= \left \| (1+\xi^2)^{\frac{s}{2}} \frac{ 2Bb \xi}{b^2+\xi^2} \widehat{u^n} \right \|_{L^2}
\le B \|u^n\|_{H^s}.
\end{align*}
Hence we obtain 
\begin{align*}
\|u^{n+1}(t)\|_{H^s} \le \mathrm{e}^{C(1+B)U^n(t)}\left(\|u_0\|_{H^s}
+C(1+B) \int_0^t \mathrm{e}^{-C(1+B)U^n(\tau)}  \|u^n(\tau)  \|_{H^s} d \tau \right).
\end{align*}

Set $ \tilde{C} = C(1+B)$ and $T=\min \left \{\frac{\log\frac{3}{2}}{2\tilde{C}\|u_0\|_{H^s}}, \, \frac{1}{4\tilde{C}} \right\}$.
We verify that 
\begin{align*}
\|u^n(t)\|_{H^s} \le 2\|u_0\|_{H^s}, \quad t \in [0,T]
\end{align*}
by mathematical induction.
The case of $ n=0$ is trivially true.
Assume that the above inequality holds for $n$.
Then we obtain
\begin{align*}
\|u^{n+1}(t)\|_{H^s} 
&\le \mathrm{e}^{2\tilde{C}\|u_0\|_{H^s}t}
(\|u_0\|_{H^s}+2 \tilde{C} t \|u_0\|_{H^s})\\
&\le 2 \|u_0\|_{H^s}.
\end{align*}
This completes the inductive step.
The remainder of the proof runs as \cite[Theorem 1.1]{holmes2017well}.
\end{proof}

\begin{pro}
\label{pro2}
Assume that $ s> \frac{3}{2}$ and $ u_0 \in H^s$. 
Also, let $T_0>0$ be the maximal existence time of the corresponding solution $u$ to \eqref{eq2}. 
If $T_0 < \infty$, then 
\begin{align}
\int_0^{T_0} \|u_x(t)\|_{L^{\infty}} dt =\infty. \notag
\end{align}
\end{pro}
\begin{proof}
As in \cite[Lemma 3.6]{koch2003local}, we have 
\begin{align*}
\| u(t) \|_{H^s}
& \le \| u_0 \|_{H^s} \mathrm{exp}\bigl(c \|u_x \|_{L^1((0,t);L^{\infty})}\bigr).
\end{align*}
Therefore if 
\begin{align*}
\int_0^{T_0} \|u_x(t)\|_{L^{\infty}} dt <\infty,
\end{align*}
$\sup_{0<t<T_0} \| u(t) \|_{H^s} <\infty $.
Then by Proposition \ref{pro1}, 
we can extend the solution $u(t)$ to $[0,T_0+\epsilon)$ for some $ \epsilon>0$.
This contradicts the definition of $T_0$.
\end{proof}

Multiplying \eqref{eq2} by $u$, integrating over $\mathbf{R}$ and using the integration by parts, we obtain
the $L^2$ conservation law $\|u(t)\|_{L^2}=\|u_0\|_{L^2}$.

Finally, we give the definition of wave-breaking.
We say that wave-breaking occurs 
if 
$ \sup_{t\in[0,T_0)} \|u(t)\|_{L^{\infty}}<\infty $ 
while $ \limsup_{t \uparrow T_0}  \|u_x(t)\|_{L^{\infty}}=\infty$ for some $ 0<T_0<\infty$.

\section{Blow-up results}
\label{sec3}
From section \ref{sec3} through section \ref{sec5}, we study \eqref{eq2}.
In this section, we give the proof of Theorem \ref{thm1} and some results on the blow-up.
We also show that the finite-time blow-up can occur only as a result of wave-breaking.

Let us consider the following differential equation:
\begin{align}
\label{eq3}
\left \{
\begin{aligned}
& \frac{dq}{dt}=u(t,q), \, \, t \in [0,T_0), \\
& q(0,x_0) = x_0, \, \, x_0 \in \mathbf{R}, 
\end{aligned}
\right.
\end{align}
where $ u $ is a solution of (\ref{eq2}) and $T_0$ is the maximal existence time of the solution $ u$.
If $ u \in C([0,T_0);H^3)$, a solution $q(t,x_0)$ can be defined on $[0,T_0)$ for any $x_0 \in \mathbf{R}$.
Moreover, $ q(t,\cdot)$ is a diffeomorphism of $\mathbf{R}$ for every $ t \in [0,T_0)$ (see \cite[Theorem 3.1]{constantin2000existence}).

Here we give a bound of $u(t,q(t))$ and of $\|u(t)\|_{L^{\infty}}$. 
We use these throughout the study of \eqref{eq2}.
\begin{lem}
\label{lem1}
Assume that $s \ge3$ and $u_0 \in H^s$.
Also, let $u$ be the corresponding solution of \eqref{eq2}.
Then for $t \in [0,T_0)$,
\begin{align}
\label{eqlem1.2}
u_0(x_0)-B \, b^{\frac{1}{2}}\| u_0 \|_{L^2} \, t \le u(t,q(t)) \le u_0(x_0)+B \, b^{\frac{1}{2}}\| u_0 \|_{L^2} \, t. 
\end{align}
Furthermore,
\begin{align}
\label{eqlem1.3}
\|u(t)\|_{L^{\infty}} \le \|u_0\|_{L^{\infty}}+B \, b^{\frac{1}{2}}\| u_0 \|_{L^2} \, t. 
\end{align}
\end{lem}
\begin{proof}
Set $U(t)=u(t,q(t))$. By the definition of $q(t)$ and the integration by parts, we have 
\begin{align}
\begin{aligned}
\label{eqlem1.4}
\frac{dU}{dt} & = u_t(t,q(t))+u(t,q(t)) \cdot  u_x(t,q(t)) \\
& = - \int _{\mathbf{R}} B \mathrm{e}^{-b|q(t)-\xi|}u_{\xi}(t,\xi) \, d \xi \\
& = \int _{\mathbf{R}} B \, b \, \mathrm{sgn}(q(t)- \xi) \mathrm{e}^{-b|q(t)-\xi|}u(t,\xi) \, d \xi,
\end{aligned}
\end{align}
where $\mathrm{sgn}(\, \cdot \,)$ is the sign function.
Since 
\begin{align}
\left | \int _{\mathbf{R}}  B \, b \, \mathrm{sgn}(q(t)- \xi) \mathrm{e}^{-b|q(t)-\xi|}u(t,\xi) \, d \xi \right | 
& \le B b \| \mathrm{e}^{-b|\cdot|} \|_{L^2} \| u(t) \|_{L^2} \notag \\ 
& = B b^{\frac{1}{2}}\| u_0 \|_{L^2}, \notag
\end{align}
it follows that
\begin{align}
\label{eqlem1.1}
- B \, b^{\frac{1}{2}} \| u_0 \|_{L^2} \le \frac{dU}{dt} \le B \, b^{\frac{1}{2}} \| u_0 \|_{L^2}.
\end{align}
By integrating (\ref{eqlem1.1}) and the definition of $U(t)$, we obtain 
\begin{align}
u_0(x_0)-B \, b^{\frac{1}{2}}\| u_0 \|_{L^2} \, t \le u(t,q(t)) \le u_0(x_0)+B \, b^{\frac{1}{2}}\| u_0 \|_{L^2} \, t. \notag
\end{align}

Moreover, for any $ x_0 \in \mathbf{R}$, we have
\begin{align}
|u(t,q(t,x_0))| & \le |u_0(x_0)|+B \, b^{\frac{1}{2}}\| u_0 \|_{L^2} \, t, \notag \\ 
& \le \|u_0\|_{L^{\infty}}+B \, b^{\frac{1}{2}}\| u_0 \|_{L^2} \, t. \notag
\end{align}
Since $ q(t,\cdot)$ is a diffeomorphism of $\mathbf{R}$ for every $ t \in [0,T_0)$, 
we obtain
\begin{align}
\|u(t)\|_{L^{\infty}} \le \|u_0\|_{L^{\infty}}+B \, b^{\frac{1}{2}}\| u_0 \|_{L^2} \, t. \notag
\end{align}
\end{proof}

By a similar method, we can get an upper bound of $\mathrm{sup}_{x \in \mathbf{R}} u_x(t,x)$.

\begin{lem}
\label{lem66}
Assume that $s \ge3$ and $u_0 \in H^s$.
Also, let $u$ be the corresponding solution of \eqref{eq2}.
Then for $t \in [0,T_0)$,
\begin{align}
\sup_{x \in \mathbf{R}}u_x(t,x) \le
\sup_{x \in \mathbf{R}}u_0'(x) + 2B b  \|u_0\|_{L^{\infty}} t + B b^{\frac{3}{2}}\| u_0 \|_{L^2}t  +   B^2 b^{\frac{3}{2}}\| u_0 \|_{L^2} t^2. 
\end{align}
\end{lem}
\begin{proof}
Set $V(t)=u_x(t,q(t))$.
Using the integration by parts and the inequality \eqref{eqlem1.2},
we obtain
\begin{align}
\begin{aligned}
\label{eqprf5}
\frac{dV}{dt} &= u_{tx}(t,q(t)) \, + \, u(t,q(t)) \cdot u_{xx}(t,q(t)) \\
& = - u_x(t,q(t))^2 \, - \, \int _{\mathbf{R}} B \mathrm{e}^{-b|q(t)-\xi|}u_{\xi \xi}(t,\xi) \, d \xi \\
& = -V^2 + 2B b  u(t,q(t)) - b^2 \int _{\mathbf{R}} B \mathrm{e}^{-b|q(t)-\xi|}u(t,\xi) \, d \xi \\
& \le 2B b  \|u_0\|_{L^{\infty}} \, + \, 2 \, B^2 b^{\frac{3}{2}}\| u_0 \|_{L^2} t + B b^{\frac{3}{2}}\| u_0 \|_{L^2}.
\end{aligned}
\end{align}
The rest of the proof runs as Lemma \ref{lem1}.
\end{proof}

The following theorem provides more precise information on the blow-up of the solution.
\begin{thm}
\label{thm43}
Let $ s \ge 3$ and $ u_0 \in H^s$.
The corresponding solution of \eqref{eq2} blows up in finite time $0 < T_0 < \infty$ if and only if 
\begin{align}
\liminf_{t \uparrow T_0} \inf_{x \in \mathbf{R}}u_x(t,x) = - \infty.
\end{align}
Also, the finite-time blow-up can occur only as a result of wave-breaking.
\end{thm}
\begin{proof}
Let $0 < T_0<\infty$ be the maximal existence time of the solution.
Then by Proposition \ref{pro2}, we have
\begin{align}
\int_0^{T_0} \|u_x(t)\|_{L^{\infty}} dt =\infty.
\end{align}
Hence, we obtain $ \limsup_{t \uparrow T_0}  \|u_x(t)\|_{L^{\infty}} = \infty$.
On the other hand, by Lemma \ref{lem66}, $ \sup_{x \in \mathbf{R}}u_x(t,x)$ is bounded in finite time.
As a result, $\liminf_{t \uparrow T_0} \inf_{x \in \mathbf{R}}u_x(t,x) = - \infty $ follows.

Conversely, using the Sobolev embedding $ H^s(\mathbf{R})  \hookrightarrow W^{1,\infty}(\mathbf{R})$, 
we obtain 
\begin{align*}
 \|u(t)\|_{W^{1,\infty}} \le C \|u(t)\|_{H^s}.
\end{align*}
Therefore $\liminf_{t \uparrow T_0} \inf_{x \in \mathbf{R}}u_x(t,x) = - \infty$ implies 
the finite-time blow-up.

By Lemma \ref{lem1}, $\sup_{x\in\mathbf{R}} u(t,x)$ is bounded in finite time.
Hence, the second claim follows.
%
\end{proof}

We now intend to employ these results and prove Theorem \ref{thm1}.

\begin{proof}[Proof of Theorem \ref{thm1}]
Let $V(t)=u_x(t,q(t))$.
In the same way as (\ref{eqprf5}), we have 
\begin{align}
\begin{aligned}
\label{eqprf1.1}
\frac{dV}{dt} 
&\le -V^2 + 2B b  u_0(x_0) \, + \, 2 \, B^2 b^{\frac{3}{2}}\| u_0 \|_{L^2} t + B b^{\frac{3}{2}}\| u_0 \|_{L^2} \\
&= -V^2 + F(t,x_0).
\end{aligned}
\end{align}
Since $F(t,x_0)$ is a monotonically increasing function of $t$, we have 
\begin{align}
\frac{dV}{dt} \le -V^2 +F(T,x_0), \quad t \in [0,T]. \notag
\end{align}
By assumption, it follows that
\begin{align}
V(0) = u_0'(x_0) \le - \alpha \left(\frac{F(T,x_0)^{\frac{1}{4}}+ \sqrt{F(T,x_0)^{\frac{ 1}{2}}
+\frac{1+\frac{1}{\alpha}}{1-\frac{1}{2\alpha}} \, \frac{4}{T}}}{2} \right)^2 < - \alpha \, F(T,x_0)^{\frac{1}{2}}. \notag
\end{align}
Hence, we obtain
\begin{align}
\label{eqprf1.2}
V(t)<V(0)<- \alpha F(T,x_0)^{\frac{1}{2}}<0, \quad t \in [0,T].
\end{align} 

Let
\begin{align}
\tilde{V}(t)=V(t)+\sqrt{-V(t)}F(T,x_0)^{\frac{1}{4}}. \notag
\end{align}
Then from (\ref{eqprf1.2}), it follows that
\begin{align}
\tilde{V}(t)=-\sqrt{-V(t)}\left(\sqrt{-V(t)}-F(T,x_0)^{\frac{1}{4}}\right) < \tilde{V}(0)<0. \notag
\end{align}
Since $ V'(t)<0$ and $V(t)< - \alpha F(T,x_0)^{\frac{1}{2}}$, we estimate
\begin{align}	
\begin{aligned}
\label{eqprf1.3}
\tilde{V}'(t) & =-V'(t) \left( \frac{1}{2} \left( \frac{F(T,x_0)^{\frac{1}{2}}}{-V(t)} \right)^{\frac{1}{2}} - 1 \right) \\
& \le (1-\frac{1}{2 \alpha})V'(t) \\
& \le -(1-\frac{1}{2 \alpha}) (V^2 - F(T,x_0)).
\end{aligned}
\end{align}
Also, we have
\begin{align}
\begin{aligned}
\label{eqprf1.4}
\tilde{V}^2(t) &=V^2(t)-V(t)F(T,x_0)^{\frac{1}{2}}+2V(t)(-V(t))^{\frac{1}{2}}F(T,x_0)^{\frac{1}{4}}\\
& \le (1+\frac{1}{\alpha})(V^2-F(T,x_0)).
\end{aligned}
\end{align}
From (\ref{eqprf1.3}) and (\ref{eqprf1.4}), it follows that
\begin{align}
\begin{aligned}
\label{eqprf1.7}
\frac{d}{dt}\left[\frac{1}{\tilde{V}(t)}\right] & =-\frac{1}{\tilde{V}^2} \cdot \tilde{V}' \\
& \ge \frac{1-\frac{1}{2\alpha}}{1+\frac{1}{\alpha}}.
\end{aligned}
\end{align}
Integrating (\ref{eqprf1.7}), we obtain 
\begin{align}
\tilde{V}(t) & \le \frac{1} {  \frac{1}{\tilde{V}(0)} + \frac{1-\frac{1}{2\alpha}}{1+\frac{1}{\alpha}} \, t  } \notag \\
& = \frac{1}{ \frac{1}{u_0'(x_0)+\sqrt{-u_0'(x_0)}F(T,x_0)^{\frac{1}{4}}} + \frac{1-\frac{1}{2\alpha}}{1+\frac{1}{\alpha}} \, t }. \notag
\end{align}
Hence, the solution blows up in finite time if the following inequality holds:
\begin{align}
\label{eqprf1.5}
T^{*}:=-\frac{1+\frac{1}{\alpha}}{1-\frac{1}{2\alpha}} \, \frac{1}{u_0'(x_0)+\sqrt{-u_0'(x_0)}F(T,x_0)^{\frac{1}{4}}}\le T.
\end{align}
	
We can rewrite (\ref{eqprf1.5}) as 
\begin{align}
\label{eqprf1.6}
-u_0'(x_0)-F(T,x_0)^{\frac{1}{4}}\sqrt{-u_0'(x_0)}-\frac{1+\frac{1}{\alpha}}{1-\frac{1}{2\alpha}} \, \frac{1}{T} \ge 0.
\end{align}
Since $ \sqrt{-u_0'(x_0)} \ge 0$, (\ref{eqprf1.6}) is equivalent to 
\begin{align}
\sqrt{-u_0'(x_0)} \ge 
\frac{F(T,x_0)^{\frac{1}{4}}+ \sqrt{F(T,x_0)^{\frac{ 1}{2}}+\frac{1+\frac{1}{\alpha}}{1-\frac{1}{2\alpha}}\, \frac{4}{T} }}{2}, \notag
\end{align}
which follows by the assumption (\ref{eqthm1.1}).
Thus $ T^* \le T$, and the proof is complete.

%
\end{proof}

\begin{cor}
\label{cor1}
Assume that $s \ge 3$ and $u_0 \in H^s$. Let $B>1$ and $0<b<1$.
Then there exists $A(u_0)>0$ such that 
if $b\le A(u_0) B^{- \frac{4}{3}}$, the corresponding solution of \eqref{eq2}
blows up in finite time.
\end{cor} 
\begin{proof}
Setting $T=\frac{1}{B^{\frac{2}{3}}b^{\frac{1}{2}}}$, we obtain 
\begin{align*}
& \left( \frac{F(T,x_0)^{\frac{1}{4}} + \sqrt{F(T,x_0)^{\frac{ 1}{2}} +  \, \frac{16}{T}  }}{2} \right)^2 \\
& \qquad  \qquad  \le C  (\|u_0\|_{L^{\infty}}^{\frac{1}{2}}+\|u_0\|_{L^2}^{\frac{1}{2}}+1) 
(B^{\frac{1}{2}} b^{\frac{1}{2}} + B^{\frac{1}{2}} b^{\frac{3}{4}} + B^{\frac{2}{3}} b^{\frac{1}{2}}) \\
& \qquad  \qquad  \le C  (\|u_0\|_{L^{\infty}}^{\frac{1}{2}}+\|u_0\|_{L^2}^{\frac{1}{2}}+1) B^{\frac{2}{3}} b^{\frac{1}{2}}
\end{align*}
for any $x_0 \in \mathbf{R}$.
Therefore we can take 
\begin{align*}
A(u_0)=\left(\frac{\inf_{x \in \mathbf{R}} u_0'(x)}{C (\|u_0\|_{L^{\infty}}^{\frac{1}{2}}+\|u_0\|_{L^2}^{\frac{1}{2}}+1)}\right)^2.
\end{align*}
\end{proof}

Next, we show the second blow-up result.

\begin{thm}
\label{88}
Assume that $s \ge 3$ and $u_0 \in H^s$. Also, let $B>0$ and $b>0$.
If there exists $ x_0 \in \mathbf{R} $ such that 
\begin{align}
2u_0(x_0)+b^{\frac{1}{2}}\|u_0\|_{L^2}<0, \notag
\end{align}
and 
\begin{align}
u_0'(x_0) \le \frac{2Bb^{\frac{1}{2}}\|u_0\|_{L^2}}{2u_0(x_0)+b^{\frac{1}{2}}\|u_0\|_{L^2}}, \notag
\end{align}
then the corresponding solution of \eqref{eq2} blows up in finite time.
Moreover, for the blow-up time $T_0$, 
\begin{align*}
T_0 \le -\frac{1}{u_0'(x_0)}. 
\end{align*}
\end{thm}

\begin{rem}
Let $\lambda \ge 1$ and $u_0(x)=- x  \mathrm{e}^{-\lambda x^2}$.  
Then we have 
\begin{align*}
\|u_0\|_{L^2}=2^{-\frac{5}{4}}\pi^{\frac{1}{4}}\lambda^{-\frac{3}{4}}, \ \ 
u_0 \left(\frac{1}{2 \sqrt{\lambda}} \right)=- \frac{1}{2} \lambda^{-\frac{1}{2}} \mathrm{e}^{-\frac{1}{4}}  \ \
\mathrm{and} \ \ u_0'\left(\frac{1}{2 \sqrt{\lambda}}\right) = -\frac{1}{2} \mathrm{e}^{-\frac{1}{4}}.
\end{align*}
Set $ x_0= \frac{1}{2 \sqrt{\lambda}} $.
If $ \lambda \ge 2^{-\frac{1}{3}} \pi^{\frac{1}{3}} b^{\frac{2}{3}}$, we have
\begin{align*}
2u_0(x_0)+b^{\frac{1}{2}}\|u_0\|_{L^2}\le - \frac{1}{2} \mathrm{e}^{-\frac{1}{4}} \lambda^{-\frac{1}{2}} <0,
\end{align*}
and hence obtain 
\begin{align*}
-\frac{2Bb^{\frac{1}{2}}\|u_0\|_{L^2}}{2u_0(x_0)+b^{\frac{1}{2}}\|u_0\|_{L^2}} \le C Bb^{\frac{1}{2}} \lambda^{- \frac{1}{4}},
\end{align*}
where $C$ is independent of $\lambda$, $B$ and $b$.
Therefore if $ \lambda \ge C (b^{\frac{3}{2}} + B^4 b^2 +1) $,
the requirments are saitsfied, which implies that the corresponding solution of \eqref{eq2}
blows up in finite time.
\end{rem}

\begin{rem}
Fix $u_0 \in H^3$ and $ x_0 \in \mathbf{R}$ such that
\begin{align}
u_0'(x_0) = \inf_{x \in \mathbf{R}} u_0'(x) \quad \mathrm{and}  \quad u_0(x_0)<0. \notag
\end{align}
Also, for $n \in \mathbf{N}$, set 
\begin{align}
u_0^{n}(x)=n^{-\frac{1}{2}}u_0(nx) \quad \mathrm{and} \quad x_n=n^{-1}x_0. \notag
\end{align}
As we will see in the proof of Theorem \ref{nthm2}, 
$ u_0^{n} $ and $x_n$ satisfy the requirements when $n$ is sufficiently large.
Then the corresponding solution blows up in finite time, and $T_0 \le -\frac{1}{\inf_{x \in \mathbf{R}} \partial_x u_0^n(x)}$ holds.
In other words, the corresponding solution of \eqref{eq2} blows up
earlier than or at the same time as the one of \eqref{eq7}.
\end{rem}

\begin{proof}[Proof of Theorem \ref{88}]
Set $T^*= -\frac{2u_0(x_0)+b^{\frac{1}{2}}\|u_0\|_{L^2}}{2Bb^{\frac{1}{2}}\|u_0\|_{L^2}}$.
As in the proof of Theorem \ref{thm1}, we obtain
\begin{align}
\frac{dV}{dt} \le -V^2+F(T,x_0), \quad t \in [0,T]. \notag
\end{align}
Since $F(T^*,x_0) \le 0$ 
by assumption, 
we get 
\begin{align}
\frac{dV}{dt} \le -V^2, \quad t \in [0,T^*]. \notag
\end{align}
This gives
\begin{align*}
T_0 \le -\frac{1}{u_0'(x_0)}. 
\end{align*}
We can justify this conclusion because $ -\frac{1}{u_0'(x_0)} \le T^*$ by assumption.
\end{proof}

The following theorem gives detailed information about the blow-up rate.
\begin{thm}
\label{794}
Assume that $s \ge 3$ and $u_0 \in H^s$. 
Also, let $ 0<T_0<\infty $ be the blow-up time of the corresponding solution $u$ to \eqref{eq2}. Then
\begin{align}
\lim_{t \uparrow T_0} \left(\inf_{x \in \mathbf{R}} \{u_x(t,x)\}(T_0-t) \right) = -1. \notag
\end{align}
\end{thm}
\begin{proof}
Set $ m(t)=\inf_{x \in \mathbf{R}} u_x(t,x)$ and 
\begin{align*}
K=2Bb\|u_0\|_{L^{\infty}}+B b^{\frac{3}{2}}\|u_0\|_{L^2} +2B^2b^{\frac{3}{2}}\|u_0\|_{L^2}T_0.
\end{align*}
Since $ u \in C^1([0,T_0);H^2)$, $ m(t)$ is almost everywhere differentiable on $[0,T_0)$ 
(see \cite[Theorem 2.1]{constantin1998wave}).
Therefore, as in the proof of Theorem \ref{thm1}, we have
\begin{align}
\label{nthm1eq1}
-m^2-K \le \frac{dm}{dt} \le -m^2+ K.
\end{align}
By Theorem \ref{thm43}, we also have
\begin{align}
\label{nthm1eq4}
\liminf_{t \uparrow T_0}m(t) = - \infty.
\end{align}
Hence we see that for $\epsilon \in (0,1)$, there exists $t_0 \in [0,T_0)$ satisfying
\begin{align}
\label{neweq}
m(t_0)<- \sqrt{K+\frac{K}{\epsilon}} < - \sqrt{\frac{K}{\epsilon}}. 
\end{align}
From (\ref{nthm1eq1}) and \eqref{neweq} it follows that
\begin{align}
\label{nthm1eq2}
m(t)<- \sqrt{K+\frac{K}{\epsilon}} < - \sqrt{\frac{K}{\epsilon}}, \quad t \in [t_0,T_0).
\end{align}
Combining (\ref{nthm1eq1}) with (\ref{nthm1eq2}), we obtain
\begin{align}
\label{nthm1eq3}
1-\epsilon \le \frac{d}{dt}\left[\frac{1}{m(t)}\right] \le 1+\epsilon.
\end{align}
By integrating (\ref{nthm1eq3}) and using (\ref{nthm1eq4}), we have
\begin{align}
(1-\epsilon)(T_0-t) \le -\frac{1}{m(t)} \le (1+\epsilon)(T_0-t), \notag
\end{align}
which implies
\begin{align}
-\frac{1}{1-\epsilon}\le m(t)(T_0-t) \le -\frac{1}{1+\epsilon}. \notag
\end{align}
Hence, 
\begin{align}
-\frac{1}{1-\epsilon} \le \liminf_{t \rightarrow T_0}m(t)(T_0-t) \le \limsup_{t \rightarrow T_0} m(t)(T_0-t)\le -\frac{1}{1+\epsilon}. \notag
\end{align}
Letting $ \epsilon \rightarrow 0 $, we complete the proof.
\end{proof}

\begin{rem}
This proof is similar in spirit to that of \cite[Theorem 5.3]{constantin2000existence}, where
Constantin considerd the Camassa-Holm equation 
\begin{align}
\label{5666}
\left \{
\begin{aligned}
& \partial_t u - \partial^3_{txx}u + 3 u \partial_xu = 2 \partial_x u \partial^2_{xx}u + u \partial^3_{xxx}u, 
\quad (t,x) \in \mathbf{R}_+ \times \mathbf{R},  \\
& u(0,x)=u_0(x),
\end{aligned}
\right.
\end{align}
and obtained the generic bound of blow-up rate.
Theorem \ref{794} enabels us to obtain the same blow-up rate 
as in \cite{constantin2000existence}.
\end{rem}

\section{Lifespan and Proof of Theorem \ref{thm10}}
\label{sec4}
In this section, we first give a lower bound of the maximal existence time.
Combining this result with Theorem \ref{thm1}, we prove Theorem \ref{thm10}.
Then we show that the solution of \eqref{eq2} converges to
the one of \eqref{eq7} as $ B \rightarrow 0$ or $b \rightarrow 0$.
\begin{thm}
\label{nthm2}
Asuume that $ u_0 \in H^s$, $s\ge3$.
Let $ T_{max} $ be the maximal existence time of the corresponding solution $u$ of \eqref{eq2},
$ m(t)= \inf_{x \in \mathbf{R}} u_x(t,x)$ and 
$ \Phi(t)= 2Bb\|u_0\|_{L^{\infty}}+B b^{\frac{3}{2}}\|u_0\|_{L^2}+2B^2b^{\frac{3}{2}}\|u_0\|_{L^2}t.$
Then
\begin{align}
\label{15}
T_{max} \ge T_{u_0}= \sup_{T>0} \min \left\{T, \, \Phi(T)^{-\frac{1}{2}}\arctan \left(-\frac{\Phi(T)^{\frac{1}{2}}}{m(0)} \right)\right\}.
\end{align}
Moreover, for any $ \epsilon>0 $, there exists a $u_0 \in H^3$ such that 
$ T_{max}(u_0) < (1+\epsilon)T_{u_0}.$
\end{thm}
\begin{proof}
To obtain a contradiction, suppose that $ T_{max} < T_{u_0} $.
Then there exists a $T>0$ such that $T_{max}<\min \left\{T, \, \Phi(T)^{-\frac{1}{2}}\arctan \left(-\frac{\Phi(T)^{\frac{1}{2}}}{m(0)} \right)\right\}$.
As in the proof of Theorem \ref{thm1}, we have
\begin{align}
\frac{dm}{dt} \ge -m^2- \Phi(T), \quad t \in [0,T_{max}). \notag
\end{align}
From this, it follows that 
\begin{align}
\label{nthm2eq1}
\frac{1}{m^2+\Phi(T)}\frac{dm}{dt} \ge -1.
\end{align}
Integrating (\ref{nthm2eq1}) on $[0,t]$, we obtain 
\begin{align}
\arctan(\Phi(T)^{-\frac{1}{2}}m(t)) \ge \arctan(\Phi(T)^{-\frac{1}{2}}m(0))-\Phi(T)^{\frac{1}{2}}t.  \notag 
\end{align}
Hence, we get
\begin{align}
\label{nthm2eq2}
-m(t) \le \Phi(T)^{\frac{1}{2}} \frac{\tan(\Phi(T)^{\frac{1}{2}}t)- 
\Phi(T)^{-\frac{1}{2}}m(0)}{1+\Phi(T)^{-\frac{1}{2}}m(0)\tan(\Phi(T)^{\frac{1}{2}}t)}, \quad  t \in [0,T_{max}).
\end{align}
Since $ T_{max} < \Phi(T)^{-\frac{1}{2}}\arctan \left(-\frac{\Phi(T)^{\frac{1}{2}}}{m(0)}\right)$, 
$ m(t)$ is bounded on $[0,T_{max})$.
This contradicts Theorem \ref{thm43}.


To show the rest of the theorem, choose initial data $u_0 \in H^3$ and $ x_0 \in \mathbf{R}$ such that
\begin{align}
m(0)=\inf_{x \in \mathbf{R}} u_0'(x) = u_0'(x_0), \quad u_0(x_0)<0. \notag
\end{align}
In addition, for $n \in \mathbf{N}$, set 
\begin{align}
u_0^{n}(x)=n^{-\frac{1}{2}}u_0(nx), \quad x_n=n^{-1}x_0. \notag
\end{align}

We will apply Theorem \ref{88} to $u_0^{n}$ and $x_n$.
We see at once that
\begin{align}
2 u_0^{n}(x_n) + b^{\frac{1}{2}}\|u_0^{n}\|_{L^2} = 2n^{-\frac{1}{2}}u_0(x_0) + b^{\frac{1}{2}} n^{-1}\|u_0\|_{L^2}. \notag
\end{align}
Also, 
\begin{align}
\partial _x u_0^n(x_n) &= n^{\frac{1}{2}}u_0'(x_0), \notag \\
\frac{2Bb^{\frac{1}{2}}\|u_0^n\|_{L^2}}{2u_0^n(x_n)+b^{\frac{1}{2}}\|u_0^n\|_{L^2}}  
&= \frac{2Bb^{\frac{1}{2}} n^{-1} \|u_0\|_{L^2}} {2 n^{-\frac{1}{2}}u_0^n(x_0)+b^{\frac{1}{2}} n^{-1}\|u_0\|_{L^2}}.\notag
\end{align}
Hence, for $n$ sufficiently large, $u_0^n$ and $x_n$ satisfy the assumptions of Theorem \ref{88}.
Since $m(0)= u_0'(x_0)$, we obtain 
\begin{align}
\label{eqlem41}
T_{max}(u_0^n) \le - \frac{1}{n^{\frac{1}{2}}m(0)}.
\end{align}

Next, set 
\begin{align}
\Phi _n(t)= 2Bb\|u_0^n\|_{L^{\infty}}+B b^{\frac{3}{2}}\|u_0^n\|_{L^2}+2B^2b^{\frac{3}{2}}\|u_0^n\|_{L^2}t.\notag
\end{align}
Then we have
\begin{align}
\Phi_n(1) & = 2Bb n^{-\frac{1}{2}}\|u_0\|_{L^{\infty}} + B b^{\frac{3}{2}} n^{-1}\|u_0^n\|_{L^2} +2B^2b^{\frac{3}{2}}n^{-1} \|u_0\|_{L^2} \notag \\
& \quad \rightarrow 0, \qquad (n \rightarrow \infty). \notag
\end{align}
Therefore, we obtain 
\begin{align}
\Phi_n(1)^{-\frac{1}{2}}\arctan \left(-\frac{\Phi_n(1)^{\frac{1}{2}}}{n^{\frac{1}{2}}m(0)} \right)
& = - \frac{1}{n^{\frac{1}{2}}m(0)} \,   \frac{\arctan \left(-\frac{\Phi_n(1)^{\frac{1}{2}}}{n^{\frac{1}{2}}m(0)} \right)} 
{-\frac{\Phi_n(1)^{\frac{1}{2}}}{n^{\frac{1}{2}}m(0)}} \notag \\
& \quad \rightarrow 0, \qquad (n \rightarrow \infty). \notag
\end{align}
In particular, $ 1 > \Phi _n(1)^{-\frac{1}{2}}\arctan \left(-\frac{\Phi _n(1)^{\frac{1}{2}}}{n^{\frac{1}{2}}m(0)} \right) $ holds for $n$ sufficiently large.
From this it follows that
\begin{align}
\label{eqlem42}
T_{u_0^n} \ge \Phi _n(1)^{-\frac{1}{2}}
\arctan \left(-\frac{\Phi _n(1)^{ \frac{1}{2} }  }   {n^{\frac{1}{2}}m(0)} \right).
\end{align} 
Combining \eqref{eqlem41} and \eqref{eqlem42}, we obtain 
\begin{align}
\frac{T_{u_0^n}}{T_{max}(u_0^n)} \ge \frac{\arctan \left(-\frac{\Phi _n(1)^{\frac{1}{2}}}{n^{\frac{1}{2}}m(0)} \right)} 
{-\frac{\Phi _n(1)^{\frac{1}{2}}}{n^{\frac{1}{2}}m(0)}}. \notag
\end{align}
Thus, we get
\begin{align}
\liminf_{n \rightarrow \infty} \frac{T_{u_0^n}}{T_{max}(u_0^n)} \ge 1. \notag
\end{align}
\end{proof}

\begin{rem}
This proof is adapted from that of \cite[Theorem 3.2]{danchin2001few},
where Danchin obtained a more simple lower bound of the maximal existence time for 
\eqref{5666}
and showed that the estimate is sharp.
To prove sharpness, he used \cite[Theorem 4.1]{constantin1998global}.
However, since the theorem strongly depends on the structure of 
\eqref{5666},
it does not seem to work for 
\eqref{eq2}.
For this reason, we use Theorem \ref{88} instead and obtain sharpness. 
%
%
%
%
\end{rem}

\begin{proof}[Proof of Theorem \ref{thm10}]
Set $m(t)= \inf_{x \in \mathbf{R}}u_x(t,x)$. 

First, we show that
\begin{align}
\label{nthm3eq1}
\limsup_{B \rightarrow 0 \, \mathrm{or} \, b \rightarrow 0} T_{max}^{B,b} \le - \frac{1}{m(0)}.
\end{align}
Fix $ \alpha \ge 1$ and $T(\alpha)>0$ sufficiently large.
Note that $m(0)<0$, and there exists $x_0 \in \mathbf{R}$ such that $m(0)=u_0'(x_0)$.
By Theorem \ref{thm1}, when $B>0$ or $b>0$ is small enough, we have 
\begin{align}
T_{max}^{B,b} \le -\frac{1+\frac{1}{\alpha}}{1-\frac{1}{2 \alpha}} \, \frac{1}{m(0)+\sqrt{-m(0)}F(T(\alpha),x_0)^{\frac{1}{4}}}, \notag
\end{align}
where $ F(t,x_0)= 2 B  b  u_0(x_0)  + B b^{\frac{3}{2}} \| u_0 \|_{L^2} +2B^2b^{\frac{3}{2}}\| u_0 \|_{L^2}t.$
From this it follows that
\begin{align}
\limsup_{B \rightarrow 0 \, \mathrm{or} \, b \rightarrow 0} T_{max}^{B,b} \le \notag
-\frac{1+\frac{1}{\alpha}}{1-\frac{1}{2 \alpha}} \frac{1}{m(0)}.
\end{align}
Letting $ \alpha \rightarrow \infty $ yields (\ref{nthm3eq1}).

Next, we show that
\begin{align}
\label{nthm3eq2}
\liminf_{B \rightarrow 0 \, \mathrm{or} \, b \rightarrow 0} T_{max}^{B,b} \ge - \frac{1}{m(0)}.
\end{align}
It is easily seen that 
\begin{align}
\lim_{\Phi(T)\rightarrow 0}\Phi(T)^{-\frac{1}{2}}\arctan \left(-\frac{\Phi(T)^{\frac{1}{2}}}{m(0)}\right) =-\frac{1}{m(0)}. \notag
\end{align}
Let $T>-\frac{1}{m(0)}$.
Then we have 
\begin{align}
T> \Phi(T)^{-\frac{1}{2}}\arctan \left(-\frac{\Phi(T)^{\frac{1}{2}}}{m(0)}\right)\notag
\end{align}
for  $\Phi(T) $ sufficiently small.
By Theorem \ref{nthm2}, it follows that 
\begin{align}
T_{max}^{B,b} \ge \Phi(T)^{-\frac{1}{2}}\arctan \left(-\frac{\Phi(T)^{\frac{1}{2}}}{m(0)}\right). \notag
\end{align}
Since $ B\rightarrow 0$ or $ b\rightarrow 0$ implies $ \Phi(T) \rightarrow 0$, 
we obtain (\ref{nthm3eq2}).
\end{proof}

\begin{lem}
\label{lem57}
Assume that $ u_0 \in H^s$, $s\ge3$.
Let $u^{B,b}$ and $v$ be the corresponding solutions of \eqref{eq2} and of \eqref{eq7}, respectively.
Then 
\begin{align*}
\| u^{B,b}-v \|_{L^{\infty}([0,T];L^{\infty})}
\le B b^{\frac{1}{2}} T  \|u_0\|_{L^2}  \mathrm{exp}\left( \|v_x\|_{L^1((0,T);L^{\infty})} \right)
\end{align*}
for any $T \in (0,-1 /\inf_{x \in \mathbf{R}}u_0'(x))$.
\end{lem}
\begin{proof}
Let $q(t,x_0)$ be the solution of the following differential equation:
\begin{align}
\left \{
\begin{aligned}
\notag
& \frac{dq}{dt}=u^{B,b}(t,q), \quad t \in [0,T_{max}^{B,b}), \\
& q(0,x_0) = x_0, \quad x_0 \in \mathbf{R}.
\end{aligned}
\right.
\end{align}
By the integration by parts, we obtain 
\begin{align}
&\frac{d}{dt}[u^{B,b}(t,q(t,x_0))-v(t,q(t,x_0))] \notag \\
& =- B \int_{\mathbf{R}} \mathrm{e}^{-b|q(t,x_0)-\xi|} u_{\xi}^{B,b}(t,\xi) d\xi + \left(v(t,q(t,x_0)) - u^{B,b}(t,q(t,x_0)) \right)v_x(t,q(t,x_0))  \notag \\
& \le B b^{\frac{1}{2}}\|u_0\|_{L^2} + \| v_x(t) \|_{L^{\infty}} \|  u^{B,b}(t) -v(t) \|_{L^{\infty}}. \notag
\end{align}
In the same way, we have 
\begin{align}
\frac{d}{dt}[v(t,q(t,x_0))-u^{B,b}(t,q(t,x_0))] \le B b^{\frac{1}{2}}\|u_0\|_{L^2} 
+  \| v_x(t) \|_{L^{\infty}} \|  u^{B,b}(t) -v(t) \|_{L^{\infty}}. \notag
\end{align}
Since $u^{B,b}(t)$ and $v(t)$ have the same initial data, and $q(t,\cdot)$ is a diffeomorphism of $\mathbf{R}$, 
integrating these inequalities on $[0,t]$ yields
\begin{align}
\| u^{B,b}(t)-v(t) \|_{L^{\infty}}
\le B b^{\frac{1}{2}}\|u_0\|_{L^2} t + \int_0^t \| v_x(s) \|_{L^{\infty}} \|  u^{B,b}(s) -v(s) \|_{L^{\infty}} ds. \notag
\end{align}
The Gronwall inequality completes the proof.
\end{proof}

\begin{thm}
Assume that $s\ge3$ and $ u_0 \in H^s$.
Let $u^{B,b}$ and $v$ be the corresponding solutions of \eqref{eq2} and of \eqref{eq7}, respectively.
Then
\begin{align*}
\lim_{B \rightarrow 0 \, \mathrm{or} \, b \rightarrow 0} \| u^{B,b}-v \|_{{L^{\infty}([0,T];H^s)}} =0
\end{align*}
for any $T \in (0,-1 /\inf_{x \in \mathbf{R}}u_0'(x))$.
\end{thm}
\begin{proof}
By Theorem \ref{lem57}, $ u^{B,b} \rightarrow v $ in $ L^{\infty}([0,T];L^{\infty})$ as $ B \rightarrow 0$ or $b \rightarrow 0$.
Also, we have $ \|u^{B,b}(t)\|_{L^2}= \|u_0\|_{L^2}$.
Therefore we see that $ u^{B,b} \rightarrow v $ weekly in $ L^2([0,T];L^2)$ as $ B \rightarrow 0$ or $b \rightarrow 0$.
Since $ L^2([0,T];L^2)$ is a Hilbert space, using $ \|u^{B,b}(t)\|_{L^2}= \|u_0\|_{L^2}$ again, 
we conclude that $ u^{B,b} \rightarrow v $ in $ L^2([0,T];L^2)$ as $ B \rightarrow 0$ or $b \rightarrow 0$.
Moreover, we see that $ u^{B,b} \rightarrow v $ in $ L^1([0,T];L^2)$ as $ B \rightarrow 0$ or $b \rightarrow 0$
because $[0,T]$ has a finite measure.

Let us examine $\|u^{B,b}(t)-v(t)\|_{L^2}$. Using the integration by parts, we obtain
\begin{align*}
 &\frac{d}{dt}\|u^{{B},{b}}(t) - v(t)\|_{L^2}^2 \\
& = -(u^{{B},{b}} u^{{B},{b}}_x -v v_x, u^{{B},{b}} - v) 
-(\int _{\mathbf{R}} B \mathrm{e}^{-b|x-\xi|}u^{{B},{b}}_{\xi}(t,\xi) \, d \xi, 
 u^{{B},{b}} - v) \\
&  \le (   \| u^{{B},{b}}_x(t) \|_{L^{\infty}} + \|v_x(t)\|_{L^{\infty}}) \| u^{{B},{b}}(t) - v(t)\|_{L^2}^2
+ 2B \|u_0\|_{L^2} \| u^{{B},{b}}(t) - v(t)\|_{L^2}.
\end{align*}
Integrating this on $[0,T]$ implies 
\begin{align*}
& \|u^{B,b}-v\|^2_{L^{\infty}([0,T];L^2)} \\
&  \le  (\| u^{{B},{b}}_x \|_{L^{\infty}([0,T];L^{\infty})} + \|v_x\|_{L^{\infty}([0,T];L^{\infty})}) \| u^{{B},{b}} - v\|_{L^{2}([0,T];L^2)}^2 \\
& \quad  + 2B \|u_0\|_{L^2} \| u^{{B},{b}} - v\|_{L^{1}([0,T];L^{2})}.
\end{align*}
Hence, if $ \| u^{{B},{b}}_x \|_{L^{\infty}([0,T];L^{\infty})}$ is bounded when $B$ or $b$ is sufficiently small, 
\begin{align*}
\lim_{B \rightarrow 0 \, \mathrm{or} \, b \rightarrow 0} \|u^{B,b}-v\|_{L^{\infty}([0,T];L^2)} =0
\end{align*}
follows.  

By Lemma \ref{lem66}, we have
\begin{align*}
\sup_{x \in \mathbf{R}}u^{B,b}_x(t,x) \le
\sup_{x \in \mathbf{R}}u_0'(x) + 2B b  \|u_0\|_{L^{\infty}} t + B b^{\frac{3}{2}}\| u_0 \|_{L^2}t  +   B^2 b^{\frac{3}{2}}\| u_0 \|_{L^2} t^2. 
\end{align*}
Therefore $\sup_{x \in \mathbf{R}}u^{B,b}_x(t,x)$ is bounded on $[0,T]$.
In the same manner as \eqref{nthm2eq2}, we have
\begin{align*}
-\inf_{x \in \mathbf{R}} u^{B,b}_x(t,x) \le 
\Phi(T)^{\frac{1}{2}} \frac{\tan(\Phi(T)^{\frac{1}{2}}T) 
- \Phi(T)^{-\frac{1}{2}}m(0)}{1+\Phi(T)^{-\frac{1}{2}}m(0)\tan(\Phi(T)^{\frac{1}{2}}T)}, \quad  t \in [0,T],
\end{align*}
where $ \Phi(T)= 2Bb\|u_0\|_{L^{\infty}}+B b^{\frac{3}{2}}\|u_0\|_{L^2}+2B^2b^{\frac{3}{2}}\|u_0\|_{L^2}T $,
and $m(0)=\inf_{x \in \mathbf{R}} u_0'(x)$. 
It is clear that 
\begin{align*}
\lim_{B \rightarrow 0 \, \mathrm{or} \, b \rightarrow 0} 
\Phi(T)^{\frac{1}{2}} \frac{\tan(\Phi(T)^{\frac{1}{2}}T)- \Phi(T)^{-\frac{1}{2}}m(0)}{1+\Phi(T)^{-\frac{1}{2}}m(0)\tan(\Phi(T)^{\frac{1}{2}}T)}
= \frac{1}{-\frac{1}{m(0)} - T}.
\end{align*}
Therefore $ - \inf_{x \in \mathbf{R}}u^{B,b}_x(t,x)$ is bounded on $[0,T]$ when $B$ or $b$ is sufficiently small.
Thus $ \| u^{{B},{b}}_x \|_{L^{\infty}([0,T];L^{\infty})}$ is bounded, which implies
\begin{align*}
\lim_{B \rightarrow 0 \, \mathrm{or} \, b \rightarrow 0} \|u^{B,b}-v\|_{L^{\infty}([0,T];L^2)} =0.
\end{align*}

For $ t \in [0,T]$, we have the following inequality (see \cite[Lemma 3.6]{koch2003local}):
\begin{align}
\| u^{B,b}(t) \|_{H^s} 
& \le \| u_0 \|_{H^s} \mathrm{exp}\bigl(c \|u^{B,b}_x \|_{L^1((0,T);L^{\infty})}\bigr) \notag \\
&\le C \| u_0 \|_{H^s}, \notag
\end{align}
where the constant $C$ is independent of $B$ and $b$.
The same inequality also holds for $v$.
Hence by the Bona-Smith approximation argument (see \cite{bona1975initial}), it follows that  
\begin{align}
\lim_{B \rightarrow 0 \, \mathrm{or} \, b \rightarrow 0} \| u^{B,b} - v\|_{L^{\infty}([0,T];H^s)} =0. \notag
\end{align}
\end{proof}

\section{Continuous dependence and the proof of Theorem \ref{thm11}}
\label{sec5}
In this section, we present the continuous dependence on $(B,b)$
and the proof of Theorem \ref{thm11}.


We start with the continuous dependence in $L^2$.

\begin{lem}
\label{lem2}
Let $s\ge3$ and $u^{\tilde{B},\tilde{b}}_0, \, u^{{B},{b}}_0 \in H^s$. 
Assume that the corresponding solutions $ u^{\tilde{B},\tilde{b}} $ and $ u^{{B},{b}} $ of (\ref{eq2}) satisfy
\begin{align}
& \| u^{\tilde{B},\tilde{b}}(t) \|_{H^s} \le 2 \| u^{\tilde{B},\tilde{b}}_0 \|_{H^s}, \quad \| u^{{B},{b}}(t) \|_{H^s} \le 2 \| u^{{B},{b}}_0 \|_{H^s} \notag
\end{align}
for $ t \in [0,t_0)$.
Then
\begin{align}
& \|u^{\tilde{B},\tilde{b}} - u^{{B},{b}}\|_{L^{\infty}([0,t_0];L^2)}^2 \notag \\
& \le \left(\|u_0^{\tilde{B},\tilde{b}} - u_0^{{B},{b}}\|_{L^2}^2 
+ 4 t_0 \|u_0^{\tilde{B},\tilde{b}}\|_{H^s} \| u_0^{{B},{b}}\|_{H^s} 
\biggl(\frac{\tilde{B}|\tilde{b}-b|}{\tilde{b}{b}}+\frac{\tilde{B}-B}{b}\biggr) \right) \notag \\
& \qquad \times \mathrm{exp} \left(2t_0 \big( \|u_0^{\tilde{B},\tilde{b}}\|_{H^s} + \| u_0^{{B},{b}}\|_{H^s} \big) \right). \notag
\end{align}
\end{lem}
\begin{proof}
It is easily seen that 
\begin{align}
& \frac{d}{dt}\|u^{\tilde{B},\tilde{b}}(t) - u^{{B},{b}}(t)\|_{L^2}^2 \notag \\
& = -( u^{\tilde{B},\tilde{b}} u^{\tilde{B},\tilde{b}}_x - u^{{B},{b}} u^{{B},{b}}_x, u^{\tilde{B},\tilde{b}} - u^{{B},{b}}) \notag \\
& \quad -(\int _{\mathbf{R}} \tilde{B} \mathrm{e}^{-\tilde{b}|x-\xi|}u^{\tilde{B},\tilde{b}}_{\xi}(t,\xi) \, d \xi 
- \int _{\mathbf{R}} B \mathrm{e}^{-b|x-\xi|}u^{{B},{b}}_{\xi}(t,\xi) \, d \xi, 
u^{\tilde{B},\tilde{b}} - u^{{B},{b}}). \notag 
\end{align}
By the integration by parts, we obtain 
\begin{align}
\begin{aligned}
\label{eqlem2.6}
(\mathrm{First} \, \, \mathrm{term} \, \, \mathrm{of} \, \, \mathrm{RHS}) 
& \le ( \| u^{\tilde{B},\tilde{b}}_x(t) \|_{L^{\infty}} +   \| u^{{B},{b}}_x(t) \|_{L^{\infty}} ) \|u^{\tilde{B},\tilde{b}}(t) - u^{{B},{b}}(t)\|_{L^2}^2 \\
& \le 2(\| u^{\tilde{B},\tilde{b}}_0 \|_{H^s} +   \| u^{{B},{b}}_0 \|_{H^s} )\|u^{\tilde{B},\tilde{b}}(t) - u^{{B},{b}}(t)\|_{L^2}^2.
\end{aligned}
\end{align}
Next, we estimate the second term. 
Note that for any $\tilde{b}$, $b>0$, we have 
\begin{align}
\label{eqlem2.5}
\|\mathrm{e}^{-\tilde{b}| \cdot |} - \mathrm{e}^{-b| \cdot |} \|_{L^1} = 2 \frac{|\tilde{b}-b|}{\tilde{b}b}
\end{align}
because 
\begin{align}
\left|\mathrm{e}^{-\tilde{b}| x |} - \mathrm{e}^{-b| x |} \right| = \mathrm{e}^{-\tilde{b}| x |} - \mathrm{e}^{-b| x |}
\quad \mathrm{or} \quad 
\left|\mathrm{e}^{-\tilde{b}| x |} - \mathrm{e}^{-b| x |} \right| = \mathrm{e}^{-b| x |} - \mathrm{e}^{-\tilde{b}| x |} \notag  
\end{align}
always holds.
Using the integration by parts, the Cauchy-Schwarz inequality, Young's inequality and the equality (\ref{eqlem2.5}), we obtain 
\begin{align}
\begin{aligned}
\label{eqlem2.7}
(\mathrm{Second} \, \, \mathrm{term} \, \, \mathrm{of} \, \, \mathrm{RHS}) &= ( u^{\tilde{B},\tilde{b}}, 
\int _{\mathbf{R}} ( \tilde{B} \mathrm{e}^{-\tilde{b}|x-\xi|}-B \mathrm{e}^{-b|x-\xi|})u^{{B},{b}}_{\xi}(t,\xi) \, d \xi) \\
& \le \| u^{\tilde{B},\tilde{b}}(t) \|_{L^2} \| u^{{B},{b}}_x(t) \|_{L^2} \| \tilde{B}\mathrm{e}^{-\tilde{b}| \cdot |} - B \mathrm{e}^{-b| \cdot |} \|_{L^1} \\
& \le 2 \| u^{\tilde{B},\tilde{b}}_0 \|_{H^s} \| u^{{B},{b}}_0 \|_{H^s} \\
& \qquad \times \left( \tilde{B} \|\mathrm{e}^{-\tilde{b}| \cdot |} - \mathrm{e}^{-b| \cdot |} \|_{L^1} 
+| \tilde{B} - B | \| \mathrm{e}^{-b| \cdot |} \|_{L^1} \right) \\
& = 4 \| u^{\tilde{B},\tilde{b}}_0 \|_{H^s} \| u^{{B},{b}}_0 \|_{H^s} 
\left(\frac{\tilde{B}|\tilde{b}-b|}{\tilde{b}b} + \frac{|\tilde{B}-B|}{b} \right)
\end{aligned}
\end{align}
Combining \eqref{eqlem2.6} and \eqref{eqlem2.7}, we get
\begin{align}
 \frac{d}{dt}\|u^{\tilde{B},\tilde{b}}(t) - u^{{B},{b}}(t)\|_{L^2}^2 \notag 
& \le 4 \| u^{\tilde{B},\tilde{b}}_0 \|_{H^s} \| u^{{B},{b}}_0 \|_{H^s} 
 \left(\frac{\tilde{B}|\tilde{b}-b|}{\tilde{b}b} + \frac{|\tilde{B}-B|}{b} \right) \notag \\
& \quad + 2(\| u^{\tilde{B},\tilde{b}}_0 \|_{H^s} +   \| u^{{B},{b}}_0 \|_{H^s} )\|u^{\tilde{B},\tilde{b}}(t) - u^{{B},{b}}(t)\|_{L^2}^2. \notag
\end{align}
The Gronwall inequality completes the proof.
\end{proof}

Next, using this lemma, we prove the continuous dependence in $H^s$.

\begin{thm}
\label{thm2}
Assume that $s\ge3$ and $u^{\tilde{B},\tilde{b}}_0, \, u^{{B},{b}}_0 \in H^s$. 
Let $T_{\mathrm{max}}^{\tilde{B},\tilde{b}}$ and $T_{\mathrm{max}}^{B,b}$ 
be the maximal existence time of the corresponding solutions $u^{\tilde{B},\tilde{b}}$ and $u^{B,b}$ of (\ref{eq2}), respevtively. 
If 
\begin{align}
\lim_{(\tilde{B},\tilde{b}) \rightarrow (B,b)}\|u_0^{\tilde{B},\tilde{b}} - u_0^{B,b}\|_{H^s} =0, \notag
\end{align}
then $T_{\mathrm{max}}^{B,b} \le \liminf_{(\tilde{B},\tilde{b}) \rightarrow (B,b)} T_{\mathrm{max}}^{\tilde{B},\tilde{b}}$, and
\begin{align}
\lim_{(\tilde{B},\tilde{b}) \rightarrow (B,b)}\|u^{\tilde{B},\tilde{b}} - u^{B,b}\|_{L^{\infty}([0,T];H^s)} =0 \notag
\end{align}
for any $ T \in (0,T_{\mathrm{max}}^{B,b}) $.
\end{thm}
\begin{proof}
Let $ T \in (0, T_{\mathrm{max}}^{B,b}) $, and fix $ t_0< 
\min \left \{\frac{C}{(1+B) \| u^{B,b} \|_{L^{\infty}([0,T];H^s)}}, \, \frac{C}{1+B}\right \}$.
By the assumption of $ u_0^{\tilde{B},\tilde{b}} $ and Proposition \ref{pro1}, 
when $ (\tilde{B},\tilde{b}) $ is close enough to $ (B,b)$,
$ u^{\tilde{B},\tilde{b}}(t) $ and $ u^{B,b}(t) $ exist on $ [0,t_0]$.
Furthermore,
$ \| u^{\tilde{B},\tilde{b}}(t) \|_{H^s} \le 2 \| u^{\tilde{B},\tilde{b}}_0 \|_{H^s}$
and
$ \| u^{B,b}(t) \|_{H^s} \le 2 \| u^{B,b}_0 \|_{H^s}$ 
hold there.
Therefore from Lemma \ref{lem2}, it follows that
\begin{align}
\lim_{(\tilde{B},\tilde{b}) \rightarrow (B,b)}\|u^{\tilde{B},\tilde{b}} - u^{B,b}\|_{L^{\infty}([0,t_0];L^2)} =0. \notag
\end{align}
Moreover, for $ t \in [0,t_0]$, we have the following inequality (see \cite[Lemma 3.6]{koch2003local}):
\begin{align}
\| u^{\tilde{B},\tilde{b}}(t) \|_{H^s} 
& \le \| u_0^{\tilde{B},\tilde{b}} \|_{H^s} \mathrm{exp}(c \|u^{\tilde{B},\tilde{b}}_x \|_{L^1((0,t_0);L^{\infty})}) \notag \\
& \le \| u_0^{\tilde{B},\tilde{b}} \|_{H^s} \mathrm{exp}(2 c \, t_0 \|u_0^{\tilde{B},\tilde{b}} \|_{H^s}). \notag
\end{align}
We also have the same inequality for $u^{B,b}$.
Note that the constant $c$ in the above estimate is independent of $(\tilde{B},\tilde{b})$.
Hence by the Bona-Smith approximation argument (see \cite{bona1975initial}), it follows that  
\begin{align}
\lim_{(\tilde{B},\tilde{b}) \rightarrow (B,b)}\|u^{\tilde{B},\tilde{b}} - u^{B,b}\|_{L^{\infty}([0,t_0];H^s)} =0. \notag
\end{align} 
In particular, $ u^{\tilde{B},\tilde{b}}(t_0) \rightarrow u^{B,b}(t_0)  \, \, \mathrm{in} \, \, H^s$.
Iterating this argument, we complete the proof.
\end{proof}

Combining this theorem with Theorem \ref{thm1}, we have Theorem \ref{thm11}.

\begin{proof}[Proof of Theorem \ref{thm11}]
We first prove that $G(u_0)$ is a closed set of $\mathbf{R}^2$.
Let $(B_n, b_n) \in G(u_0)$ and $(B_n, b_n) \rightarrow (B,b) \in [0,\infty)\times [0,\infty)$.

Suppose that $B=0$ or $b=0$.
Then $B_n$ or $b_n$ is close to zero for $n$ sufficiently large.
Also, $B_n$ and $b_n$ are bounded.
Therefore by Theorem \ref{thm1}, 
the solution with $(B_n,b_n)$ blows up in finite time for $n$ sufficiently large.
In other words, $(B_n, b_n) \notin G(u_0)$.
This is a contradiction. 
Thus, $(B,b) \in (0,\infty) \times(0,\infty)$.

Next, suppose that $(B,b) \in B(u_0)$. 
By Theorem \ref{thm43}, we have
\begin{align}
\label{eqthm3.2}
\liminf_{t \uparrow T_{max}^{B,b}} \inf_{x \in \mathbf{R}} u_x^{B,b}(t,x) = - \infty.
\end{align}
On the other hand, by Theorem \ref{thm2}, we have
\begin{align}
\label{eqthm3.1}
\lim_{n \rightarrow \infty} \| u^{B_n, b_n} - u^{B,b} \|_{L^{\infty}([0,T];H^s)}=0
\end{align} 
for any $T < T_{max}^{B,b} $.
In general, $ | \inf f - \inf g| \le \| f-g \|_{L^{\infty}}$ holds. 
Using this inequality, Sobolev's inequality and (\ref{eqthm3.1}), we obtain
\begin{align}
\begin{aligned}
\label{eqthm3.3}
| \inf_{x \in \mathbf{R}}u_x^{B_n,b_n}(t,x) - \inf_{x \in \mathbf{R}} u_x^{B,b}(t,x)| & \le \| u_x^{B_n,b_n}(t) - u_x^{B,b}(t) \|_{L^{\infty}} \\
& \le \| u^{B_n,b_n} - u^{B,b} \|_{L^{\infty}([0,T];H^s)} \\
& \rightarrow 0, \quad (n \rightarrow \infty)
\end{aligned}
\end{align}
for any t $ \in [0, T] $.
From (\ref{eqthm3.2}) and (\ref{eqthm3.3}), 
we can take  a subsequence $(B_{n(k)},b_{n(k)})$ and $ t_k<T_{max}^{B,b}$ such that
for all $ k \in \mathbf{N}$,
\begin{align*}
 \inf_{x \in \mathbf{R}} {u_x^{B_{n(k)},b_{n(k)}}(t_k,x)}<-k.
\end{align*}
%
%
Without loss of generality, we can assume that $(B_n,b_n)$ itself and $t_n<T_{max}^{B,b}$ satisfy 
\begin{align}
\label{eqthm3.4}
\inf_{x \in \mathbf{R}} {u_x^{B_n,b_n}(t_n,x)}<-n
\end{align}
for all $ n \in \mathbf{N}$.

We will apply Theorem \ref{thm1} to $ u^{B_n,b_n}(t_n)$.
Let 
\begin{align*}
F_n(t,x_0)=2 B_n  b_n  u^{B_n,b_n}(t_n,x_0)  + B_n b_n^{\frac{3}{2}} \| u_0 \|_{L^2}+ 2 B_n^2 b_n^{\frac{3}{2}} \| u_0 \|_{L^2} t.
\end{align*}
Using the inequality (\ref{eqlem1.3}), we obtain
\begin{align*}
&F_n(t,x_0) 
\le 2 B_n  b_n \| u_0 \|_{L^{\infty}} + B_n b_n^{\frac{3}{2}} \| u_0 \|_{L^2} +  2 B_n^2 b_n^{\frac{3}{2}} \|u_0\|_{L^2} t_n + 2 B_n^2 b_n^{\frac{3}{2}} \| u_0 \|_{L^2} t \notag 
\end{align*}
for any $x_0 \in \mathbf{R}$.
Set $t = T_{max}^{B,b}$.
Since  $B_n$ and $b_n$ are bounded, and $ t_n<T_{max}^{B,b}$, we can write
\begin{align}
\label{eq759}
F_n(T_{max}^{B,b},x_0) \le \gamma_1 + \gamma_2 \, T_{max}^{B,b}  
\end{align}
for some $ \gamma_1, \gamma_2 >0$.
From (\ref{eqthm3.4}) and \eqref{eq759} it follows that
\begin{align}
\inf_{x \in \mathbf{R}}{u_x^{B_n,b_n}(t_n,x)}< -  \left(\frac{F_n(T_{max}^{B,b},x_0)^{\frac{1}{4}}+ \sqrt{F_n(T_{max}^{B,b},x_0)^{\frac{ 1}{2}}
+\frac{16}{T_{max}^{B,b}}  }}{2} \right)^2 \notag 
\end{align}
for $n$ sufficiently large.
Thus by Theorem \ref{thm1}, $u^{B_n,b_n}$ blows up in finite time, which implies $(B_n,b_n) \in B(u_0)$.
This contradicts our assumption.
Hence $(B,b) \in G(u_0)$.

As a result, $ G(u_0) $ is a closed set of $\mathbf{R}^2$.
Since $ B(u_0) = \left( G(u_0) \cup ((0,\infty)\times(0,\infty))^{\mathrm{c}} \right)^{\mathrm{c}}$,
$B(u_0)$ is an open set of $\mathbf{R}^2$.
\end{proof}

\begin{thm}
Assume that $s\ge3$ and $u^{\tilde{B},\tilde{b}}_0, \, u^{{B},{b}}_0 \in H^s$. 
Let $T_{\mathrm{max}}^{\tilde{B},\tilde{b}}$ and $T_{\mathrm{max}}^{B,b}$ 
be the maximal existence time of the corresponding solutions $u^{\tilde{B},\tilde{b}}$ and $u^{B,b}$ of (\ref{eq2}), respevtively. 
If 
\begin{align}
\lim_{(\tilde{B},\tilde{b}) \rightarrow (B,b)}\|u_0^{\tilde{B},\tilde{b}} - u_0^{B,b}\|_{H^s} =0, \notag 
\end{align}
then 
\begin{align}
\lim _{(\tilde{B},\tilde{b}) \rightarrow (B,b)} T_{\mathrm{max}}^{\tilde{B},\tilde{b}} =T_{\mathrm{max}}^{B,b}  \notag 
\end{align}

\end{thm}
\begin{proof}
If $ T_{\mathrm{max}}^{B,b} = \infty $, the claim immediately follows from Theorem \ref{thm2}.

We consider the case of $ T_{\mathrm{max}}^{B,b} < \infty $.
Fix $(B_n,b_n) \in \mathbf{R}_+ \times \mathbf{R}_+$ such that $(B_n,b_n) \rightarrow (B,b) $.
As in the proof of Theorem \ref{thm11}, 
we can assume, without loss of generality, that
there exists $ t_n<T_{max}^{B,b}$ satisfying
\begin{align}
\label{pr}
\inf_{x \in \mathbf{R}} {u_x^{B_n,b_n}(t_n,x)}<-n
\end{align}
for all $ n \in \mathbf{N}$,
and 
can estimate 
\begin{align}
\label{pr1}
F_n(T_{max}^{B,b},x_0) \le \gamma_1 + \gamma_2 \, T_{max}^{B,b} 
\end{align}
for some $ \gamma_1, \gamma_2 >0$.
Then, for $n$ sufficiently large, we have
\begin{align}
\inf_{x \in \mathbf{R}}{u_x^{B_{n},b_{n}}(t_n,x)}< -  \left(\frac{F_{n}(T_{max}^{B,b},x_0)^{\frac{1}{4}}+ \sqrt{F_{n}(T_{max}^{B,b},x_0)^{\frac{ 1}{2}}
+\frac{16}{T_{max}^{B,b}}  }}{2} \right)^2. \notag 
\end{align}
From Theorem \ref{thm1}, \eqref{pr} and \eqref{pr1}, we obtain 
\begin{align*}
T_{max}^{B_{n},b_{n}} 
 \le T_{max}^{B,b} + \frac{C}{n},
\end{align*}
where the constant $C$ is independent of $n$.
Thus we get 
\begin{align*}
\limsup_{n \rightarrow \infty} T_{max}^{B_{n},b_{n}}  \le T_{max}^{B,b}.
\end{align*}
Combining this with Theorem \ref{thm2}, we complete the proof.
\end{proof}

\section{Proof of Theorem \ref{thm7}}
\label{sec6}
In this section, we give the proof of Theorem \ref{thm7}. 

%
%
%
%
%
%
%


\begin{proof}[Proof of Theorem \ref{thm7}]
Let $ \delta >0$ be sufficiently small.
We define  
\begin{align}
\mathcal{A} = \left \{ 
\begin{aligned}
(q,r)  \in (2,\infty] &\times (2,\infty) \, :  \, \frac{1}{q}+\frac{1}{n} \cdot \frac{2}{r} \le \frac{1}{n}-\delta; \notag \\
&\frac{1}{q} \le \frac{p^*-1}{n} -1 -(p^*-1)\delta; \, r<\frac{2}{n}(p-1)/\delta \notag
\end{aligned}
\right \}. \notag
\end{align}
For simplicity we write $ \| \cdot \|_{L^{q}_T;X}$ instead of $ \| \cdot \|_{L^{q}((0,T);X)}$.
Let $ \alpha \ge \frac{1}{2}$.
In the same manner as \cite[pp. 2612-2613]{stefanov2010well}, we obtain
\begin{align}
\begin{aligned}
\label{111}
\|u\|_{L^{q}_T;W^{\alpha,r}} &\le C( \|u_0\|_{H^{\alpha+\frac{1}{2}}} 
+\|u_0\|_{W^{\alpha+ \gamma(1-\frac{2}{r}) ,r'}} ) \\
& \quad +C \| u \|_{L^{\infty}_T;H^{\alpha + 1 + \mathrm{max}\{\frac{1}{2},\gamma \}}} 
\Big(\|u\|^{p-1}_{L^{\infty}_T;L^{\infty}_x}+\|u\|^{p-1}_{L^{(p-1)\beta}_T;L^{(p-1)\frac{2r}{r-2}}}\Big),
\end{aligned}
\end{align}
where $ 1+ \frac{1}{q} = \frac{1}{n}(1-\frac{2}{r})+\frac{1}{\beta}$.

We show that
$ \left( (p-1)\beta,(p-1)\frac{2r}{r-2} \right) \in \mathcal{A} $ for $ p \ge p^* $.
First, we have
\begin{align}
\frac{1}{(p-1)\beta}+\frac{1}{n} \cdot \frac{2}{(p-1)\frac{2r}{r-2}} &= \frac{1}{p-1}\Big(\frac{1}{\beta}+\frac{1}{n} \cdot (1-\frac{2}{r})\Big) \notag \\
& = \frac{1}{p-1}\Big(1+\frac{1}{q}\Big) \notag \\
& \le \frac{1}{p-1}\Big(\frac{p^*-1}{n}-(p^*-1)\delta \Big) \notag \\
& \le \frac{1}{n}-\delta. \notag
\end{align}
Next, we have
\begin{align}
\frac{1}{(p-1)\beta}&=\frac{1}{p-1}\Big(1+\frac{1}{q}-\frac{1}{n}(1-\frac{2}{r})\Big) \notag \\
& \le \frac{1}{p-1}(1-\delta) \notag \\
& \le \frac{1}{p^*-1}(1-\delta).\notag
\end{align}
Here if we get
\begin{align}
\label{1}
\frac{1}{p^*-1}(1-\delta)\le \frac{p^*-1}{n}-1-(p^*-1)\delta,
\end{align}
the desired inequality follows.
Note that \eqref{1} is equivalent to 
\begin{align}
\label{2}
\Big((p^*-1)-\frac{1}{p^*-1}\Big)\delta \le \frac{p^*-1}{n}-1-\frac{1}{p^*-1}.
\end{align}
By the definition of $p^*$, we have
\begin{align}
\frac{p^*-1}{n}-1-\frac{1}{p^*-1}>0. \notag
\end{align}
Therefore, $ \delta>0$ sufficiently small yields \eqref{2}.
Finally, 
\begin{align}
(p-1)\frac{2r}{r-2}&  \le \frac{2}{n}(p-1)\frac{1}{\frac{1}{n}(1-\frac{2}{r})} \notag \\
& \le \frac{2}{n}(p-1)/\delta. \notag
\end{align}
Thus, $ \left( (p-1)\beta,(p-1)\frac{2r}{r-2} \right) \in \mathcal{A} $.

We define 
\begin{align}
\| u \|_{T,\mathcal{A},\alpha}:= \sup_{(q,r)\in \mathcal{A}} \|u \|_{L^q_T;W^{\alpha,r}}. \notag
\end{align}
Since $1<r'<2$,
by an interpolation inequality, we obtain
\begin{align}
\label{222}
\|u_0\|_{H^{\alpha+\frac{1}{2}}}+\|u_0\|_{W^{\alpha+\gamma(1-\frac{2}{r}),r'}}
\le \|u_0\|_{H^{\alpha+\mathrm{max}\{\frac{1}{2},\gamma \}}}+\|u_0\|_{W^{\alpha+\gamma,1}}. 
\end{align}
As in the above argument, we can show that $ (\infty,4) \in \mathcal{A}$ 
for $ \delta>0$ sufficiently small.
Hence, we obtain
\begin{align}
\label{333}
\|u\|_{L^{\infty}_T;L^{\infty}} \le \|u\|_{L^{\infty}_T;W^{\frac{1}{2},4}} \le \|u\|_{T,\mathcal{A},\alpha} 
\end{align}
for $ \alpha \ge \frac{1}{2}$. 
Because of $ \left( (p-1)\beta,(p-1)\frac{2r}{r-2} \right) \in \mathcal{A} $, we have
\begin{align}
\label{444}
\|u\|_{L^{(p-1)\beta}_T;L^{(p-1)\frac{2r}{r-2}}} \le \| u \|_{T,\mathcal{A},\alpha}.
\end{align}
Combining \eqref{111}, \eqref{222}, \eqref{333} and \eqref{444}, we get
\begin{align}
\begin{aligned}
\label{3}
\|u\|_{T,\mathcal{A},\alpha} \le C ( \|u_0\|_{H^{\alpha+\mathrm{max}\{\frac{1}{2},\gamma \}}}
& +\|u_0\|_{W^{\alpha+\gamma,1}} \\
&  +\| u \|_{L^{\infty}_T;H^{\alpha + 1 + \frac{1}{2}}} \|u\|^{p-1}_{T,\mathcal{A},\alpha}).
\end{aligned}
\end{align}

Let $ \alpha=\frac{3}{2}$ and 
$ \|u_0\|_{H^{\frac{5}{2}+\mathrm{max}\{\frac{1}{2},\gamma\}}}+\|u_0\|_{W^{\frac{3}{2}+\gamma,1}}< \epsilon $.
Also, set $s=\frac{5}{2}+\mathrm{max}\{\frac{1}{2},\gamma\}$.
We prove the boundness of $ \| u(t) \|_{H^s} $ by contradiction.
This shows that the solution with initial data $ u_0$ is global.

Assume that $ \| u(t) \|_{H^s} $ is unbounded.
We define
\begin{align}
T^* = \inf\{t>0:\|u(t)\|_{H^s} = 4 \epsilon \}. \notag
\end{align}
Clearly, $ 0<T^*<\infty $ and $ \| u\|_{L^{\infty}_{T^*}; H^s} \le 4 \epsilon$.
From \eqref{3}
it follows that
\begin{align}
\|u\|_{T,\mathcal{A},\alpha} \le C ( \epsilon + \epsilon \|u\|^{p-1}_{T,\mathcal{A},\alpha}) \notag
\end{align}
for all $ 0<T<T^*$.
By a continuous argument, if $ \epsilon $ is sufficiently small, 
we obtain $ \|u\|_{T,\mathcal{A},\alpha} \le 2 C \epsilon $.
Note that 
\begin{align}
\|u(t)\|_{H^s} \le \|u_0\|_{H^s} \mathrm{exp}(c \|u\|^{p-1}_{L^{p-1}_T;W^{1,\infty}}) \notag 
\end{align}
holds as far as the solution exists (see \cite[p. 2611]{stefanov2010well}).
Therefore, we have
\begin{align}
\|u(t)\|_{H^s} \le 3 \epsilon \notag
\end{align}
as long as $\|u\|^{p-1}_{L^{p-1}_T;W^{1,\infty}} \le \frac{1}{c}$.
Using Sobolev's inequality, we obtain
\begin{align}
\|u\|_{L^{p-1}_T;W^{1,\infty}} \le C \|u\|_{L^{p-1}_T;W^{\frac{3}{2},\tilde{r}}} \notag
\end{align}
for $ 2 < \tilde{r} < \infty$.
Here we can choose $ \tilde{r} \in (2,\infty)$ such that $ ((p-1),\tilde{r}) \in \mathcal{A}$.
Indeed, as can be seen from the argument above, we have
\begin{align}
\frac{1}{p-1}\le \frac{1}{p^*-1} \le \frac{p^*-1}{n} -1 -(p^*-1)\delta \notag
\end{align}
for $ \delta>0$ sufficiently small.
Therefore, what is left is to find $ \tilde{r} \in (2,\infty)$ satisfying
\begin{align}
\frac{n}{2(p-1)}\delta<\frac{1}{\tilde{r}}\le \frac{1}{2}-\frac{n}{2(p^*-1)}-\frac{n}{2}\delta. \notag
\end{align}
It is easily seen that this is equivalent to
\begin{align}
\Big(\frac{n}{2(p-1)}+\frac{n}{2}\Big)\delta<\frac{1}{2}-\frac{n}{2(p^*-1)}. \notag
\end{align}
Since $ \frac{1}{2}-\frac{n}{2(p^*-1)}>0$ by the definition of $p^*$, 
this inequality holds for $ \delta>0 $ sufficiently small.
Thus, the claim follows.
Then we have
\begin{align}
\|u\|^{p-1}_{L^{p-1}_T;W^{1,\infty}} \le C \|u\|^{p-1}_{L^{p-1}_T;W^{\frac{3}{2},\tilde{r}}}
\le \|u\|^{p-1}_{T,\mathcal{A},\alpha} \le (2 C \epsilon )^{p-1}. \notag
\end{align}
Hence, taking $ \epsilon>0 $ such that $ (2 C \epsilon )^{p-1}<\frac{1}{c}$, we obtain
\begin{align}
\|u(t)\|_{H^s} \le 3 \epsilon \notag
\end{align}
for all $0<T<T^*$.
This contradicts the definition of $ T^* $.
Therefore, $ \sup_{0<t<\infty}\|u(t)\|_{H^s} $ is bounded and $ u(t) $ is a global solution.
Moreover,
\begin{align}
\sup_{0<t<\infty}\|u(t)\|_{H^s} \le 4 \epsilon \notag
\end{align}
holds.
\end{proof}

\section{Decay estimate and the proof of Theorem \ref{thm8}}
\label{sec7}
In this section, we give a decay estimate for the free solution of the generalized Fornberg-Whitham equation
\begin{align}
\label{eq4}
\left \{
\begin{aligned}
& \partial_t u  + \int _{\mathbf{R}} B \mathrm{e}^{-b|x-\xi|}u_{\xi}(t,\xi) \, d \xi =0, \\
& u(0,x)=u_0(x). 
\end{aligned}
\right.
\end{align}
It is clear that the solution of (\ref{eq4}) is 
\begin{align}
T^{B,b}(t)u_0 = \mathcal{F}^{-1}\left[\mathrm{exp}\left(-\, i \, t \, \frac{2Bb \, \xi}{b^2+\xi^2}\right) \, \widehat{u_0}\, \right]. \notag
\end{align}
\begin{lem}
\label{eqlem2.1}
Let $\mu(\xi)= \frac{\xi}{1+\xi^2}$. Then 
\begin{align}
\mu^{(n)}(\xi) = \frac{(-1)^n \, n! \, \sum^{[\frac{n+1}{2}]}_{k=0} \, (-1)^k \, \binom{n+1}{2k} \, \xi^{n+1-2k}}{(1+\xi^2)^{n+1}}. \notag
\end{align}
\end{lem}
\begin{proof}
By mathematical induction, this lemma immediately follows.
\end{proof}

\begin{lem}[Van Der Corput]
\label{lem80}
Let $k \ge 2$ be an integer, $ \psi \in C_0^{\infty}(\mathbf{R})$ and $ \mu \in C^k(\mathbf{R})$ satisfy 
$ \mu^{(k)}(\xi) > \lambda >0$ on the support of $\psi$. Then
\begin{align}
\left| \int \mathrm{e}^{i \mu(\xi)}  \psi(\xi) d\xi \right| 
\le C \lambda^{-\frac{1}{k}} \left( \|\psi\|_{L^{\infty}}+ \|\psi'\|_{L^1} \right). \notag
\end{align}
\end{lem}
\begin{proof}
See \cite[p. 334]{stein1993harmonic}.
\end{proof}

The following result is motivated by \cite[Theorem 3]{stefanov2010well}.
\begin{thm}
\label{54}
\begin{align}
\label{eqthm4.1}
& \|T^{B,b}(t)u_0\|_{L^2}=\|u_0\|_{L^2}, \\
\label{eqthm4.2}
& \|T^{B,b}(t)u_0\|_{L^r} \le C_{B,b} \, |t|^{-\frac{1}{3}(1-\frac{2}{r})} \|u_0\|_{W^{\frac{3}{2}(1-\frac{2}{r}),r'}}, \quad (2<r<\infty).
\end{align}
\end{thm}
\begin{proof}
We check at once \eqref{eqthm4.1}.
It remains to prove \eqref{eqthm4.2}.

We first consider the case of $B=\frac{1}{2}$ and $b=1$.
For simplicity, we write $T(t)$ instead of $T^{\frac{1}{2},1}(t).$
Let $j \in \mathbf{Z}_+$ and $ \varphi_j(\xi)$ be the Littlewood-Paley functions.
By Lemma \ref{eqlem2.1}, we have
\begin{align}
& \mu''(\xi) = \frac{2 (\xi^3-3\xi)}{(1+\xi^2)^3}, \notag \\
& \mu'''(\xi) =\frac{-6(\xi^4-6\xi^2+1)}{(1+\xi^2)^4}.\notag
\end{align}
Hence, for $|\xi|>\sqrt{6}$,
we obtain
\begin{align}
|\mu''(\xi) |\ge \frac{1}{8}|\xi|^{-3}. \notag
\end{align}
Applying Lemma \ref{lem80}, 
we get 
\begin{align}
\sup_{x \in \mathbf{R}} \left| \int \mathrm{e}^{ix \xi} \mathrm{e}^{-it \mu(\xi)} \varphi_j(\xi) d\xi \right| 
\le C \, 2^{\frac{3}{2}j} \, t^{-\frac{1}{2}} \notag
\end{align}
for $ j \ge 4$.
On the other hand, 
since $ \mu''(0)= \mu''(\pm \sqrt{3})=0$, 
we can not apply Lemma \ref{lem80} with $k=2$ to the neighborhoods of $\xi =0, \pm \sqrt{3}$.
However, $ \mu'''(0) \neq 0$ and $\mu'''(\pm \sqrt{3}) \neq 0$ hold.
Therefore by Lemma \ref{lem80} with $k=2$ and with $k=3$, we obtain 
\begin{align}
\sup_{x \in \mathbf{R}} \left| \int \mathrm{e}^{ix \xi} \mathrm{e}^{-it \mu(\xi)} \varphi_j(\xi) d\xi \right| \notag
\le C \, t^{-\frac{1}{2}}
\end{align}
or
\begin{align}
\sup_{x \in \mathbf{R}} \left| \int \mathrm{e}^{ix \xi} \mathrm{e}^{-it \mu(\xi)} \varphi_j(\xi) d\xi \right| \notag
\le C \, t^{-\frac{1}{3}}
\end{align}
for $ 0 \le j \le 3$.
Using the Riesz-Thorin interpolation theorem and the Littlewood-Paley theory, 
we prove \eqref{eqthm4.2} with the special case of $B=\frac{1}{2}$ and $b=1$ (see \cite[Theorem 3]{stefanov2010well} for more details).
%
%
%

Since
\begin{align}
T^{B,b}(t)[u_0(\cdot)](x)= T(2Bt)[u_0(b^{-1}\, \cdot \, )](bx), \notag
\end{align}
we complete the proof.
\end{proof}

\begin{proof}[Proof of Theorem \ref{thm8}]
The proof of Theorem \ref{thm8}
immediately follows from Theorem \ref{thm7} and Theorem \ref{54}
\end{proof}

\section*{Acknowledgements}
I would like to thank Professor Hideo Takaoka for useful advice and persistent help.
I also thank Professor Nao Hamamuki for continued support and considerable encouragement.
I am grateful to Professor Hiroaki Aikawa for valuable comments.

\bibliographystyle{siam}
\bibliography{itasaka_paper}

\end{document}